\documentclass[11pt,dvipsnames,letterpaper]{amsart}

\usepackage{style}
\usepackage[left=3.4cm,right=3.4cm,bottom=3.5cm,top=3.4cm]{geometry}

\usepackage{times}


\tikzstyle{nodal}=[circle,draw,fill=black,inner sep=0pt, minimum width=4pt]
\tikzstyle{half-fiber}=[rectangle,draw=black,thick,inner sep=0pt, minimum width=5pt, minimum height=5pt]
\tikzset{double distance = 2pt}

\author{Fabio Bernasconi}
\address{Dipartimento di Matematica “Guido
	Castelnuovo”, SAPIENZA Università di Roma, Piazzale Aldo Moro 5, I-00185 
	Roma}
\email{fabio.bernasconi@uniroma1.it}

\author{Gebhard Martin}
\address{Mathematisches Institut \\ Universität Bonn \\ Endenicher Allee 60 \\ 53115 Bonn \\ Germany}
\email{gmartin@math.uni-bonn.de}

\title{Bounding geometrically integral del Pezzo surfaces}
\date{\today}

\begin{document}

\begin{abstract}
We prove several boundedness statements for geometrically integral normal del Pezzo surfaces $X$ over arbitrary fields.

We give an explicit sharp bound on the irregularity if $X$ is canonical or regular.
In particular, we show that wild canonical del Pezzo surfaces exist only in characteristic $2$.
As an application, we deduce that canonical del Pezzo surfaces form a bounded family over $\mathbb{Z}$, generalising work of Tanaka.

More generally, we prove the BAB conjecture on the boundedness of $\varepsilon$-klt del Pezzo surfaces over arbitrary fields of characteristic different from $2, 3,$ and $5$.
\end{abstract}
	
	\maketitle
	
	\tableofcontents

\section{Introduction}

We work over a field $k$ of prime characteristic $p>0$. 
When running the Minimal Model Program (MMP for short) for klt projective varieties $Z$ with canonical divisor $K_Z$ not pseudo-effective, the outcomes are Mori fibre spaces, \emph{i.e.} projective fibrations $f \colon X \to B$ of relative Picard rank 1 where $X$ has klt singularities, $\dim B < \dim X$ and the anti-canonical divisor $-K_X$ is $f$-ample. 
It is then natural to study the geometry of $X$ in terms of the base $B$ and the general fibre.
In characteristic $p>0$, the theorem on generic smoothness on general fibres does not always hold, and there are examples of Mori fibre spaces where the general fibre might fail to be normal or even reduced \cite{MS03}. 
In this case, it is natural to study the \emph{generic} fibre $X_{k(B)}$, which is a klt Fano variety defined over the fraction field $k(B)$, which is \emph{imperfect} as soon as $\dim(B) \geq 1$.

Thanks to the recent development of the 3-dimensional MMP \cite{HX15, CTX15, BW17, HW22, Wal23}, Mori fibre spaces are known to exist for 3-folds over fields of characteristic $p > 5$. 
The next step in the classification problem consists in understanding the generic fibre of a 3-fold Mori fibre space. This work is motivated by the following general question:

\begin{Question*}
 Do the generic fibers of Mori fibre spaces form a bounded family? 
 Can we give explicit bounds on their cohomological invariants?
\end{Question*}

The main invariant we are interested in is the \emph{irregularity} of the generic fibre. Recall that the irregularity of $X_{k(B)}$ is defined as $h^1(X_{k(B)},\mathcal{O}_{X_{k(B)}}):=\dim_{k(B)} H^1(X_{k(B)}, \mathcal{O}_{X_{k(B)}})$. 
The case of relative dimension 1 is easy to treat: regular Fano curves are conics, they have vanishing irregularity and they fail to be geometrically regular only in characteristic $p=2$.
The case of relative dimension 2, \emph{i.e.} the geometry of del Pezzo surfaces over imperfect fields has turned out to be more difficult to handle. 
There are two known series of examples of canonical del Pezzo surfaces with positive irregularity:

\begin{enumerate}
    \item In \cite{Sch07}, Schr\"{o}er constructs a canonical del Pezzo surface $X$ with a unique singular factorial point of type $A_1$, $h^1(X, \mathcal{O}_X)=1, \rho(X)=1$ and $K_X^2=1$ over an arbitrary imperfect field of characteristic 2. 
    % This shows that the assumption on the characteristic in \autoref{p-sharp-tameness-3} is optimal.
    \item In \cite{Mad16} Maddock constructs regular del Pezzo surfaces $X_1$ and $X_2$ defined over 
    an imperfect field of $p$-degree $3$ (resp. $4$) with $K_{X_d}^2 = d$ and $h^1(X_d,\mathcal{O}_{X_d}) = 1$. Moreover, $X_1$ is geometrically integral and $X_2$ is not.
    % \item Note that \autoref{p-sharp-tameness-3} cannot be extended to the klt case in characteristic 3 by \cite{Tan20}.
\end{enumerate} 

On the positive side, the recent works \cite{PW,FS20, JW21, BT22} indicate that the pathological behaviour of del Pezzo fibrations is particular to small characteristics.
In this article, we further restrict the possibilities for the irregularity of geometrically integral canonical del Pezzo surfaces defined over imperfect fields.
Our first main result is the following:

\begin{theorem}\label{t-intro-irr}
	Let $X$ be a geometrically integral normal locally complete intersection del Pezzo surface over a field $k$ of characteristic $p$. If $h^1(X, \mathcal{O}_X) \neq 0$, then $k$ is an imperfect field, $\rho(X)=1$, and either
	\begin{enumerate}
		\item $p=3$, $h^1(X, \mathcal{O}_X)=2$, $K_X^2 = 1$, and $X$ is not canonical, or
		\item  $p=2$, $h^1(X, \mathcal{O}_X)=1$, and $K_X^2 \leq 2$.
	\end{enumerate}
\end{theorem}

We note that our bound on the irregularity in the regular case is sharp, as Maddock's example shows.
In \autoref{p-wildp=2}, we describe torsors over the regular wild del Pezzo surfaces in characteristic $p = 2$. This is a first step towards a classification of wild regular del Pezzo surfaces. In particular, it may be useful for the construction of explicit examples in the style of Maddock.
Note that the hypothesis on geometric integrality is automatically satisfied for normal del Pezzo surfaces appearing as generic fibres of 3-folds by \cite[Theorem 2.3]{Sch10}.

In the second part of this article, we prove boundedness results for del Pezzo surfaces over imperfect fields.
The Borisov--Alexeev--Borisov 
 (BAB) conjecture (see \autoref{BAB-conj}) states that mildly singular ($\varepsilon$-klt) Fano varieties of dimension $d$ form a bounded family over $\Spec \mathbb{Z}$. 
While the conjecture has been proven over fields of characteristic 0 by Birkar \cite{Bir21}, it is still open over fields of characteristic $p$. 

More precisely, while the case of del Pezzo surfaces over perfect fields has been known for a long time (see \cite{Ale94, AM04} and \cite[Lemma 3.1]{CTW17}), already the boundedness of $3$-dimensional Fano varieties is open. 
In this direction, the BAB conjecture for generic fibres of Mori fibre spaces would be desirable.
In \cite{Tan19}, Tanaka showed that geometrically integral regular del Pezzo surfaces form a bounded family.
Using \autoref{t-intro-irr} and the results on the irregularity of klt del Pezzo surfaces of \cite{BT22}, we are able to prove various instances of the BAB conjecture, following the strategy of Alexeev--Mori \cite{AM04}:

\begin{theorem}\label{t-intro-bdd}
The following classes of del Pezzo surfaces are bounded over $\Spec \mathbb{Z}$:
\begin{eqnarray*}
\mathcal{X}_{\textup{dP}, \textup{can}} &=& \left\{ X \mid X \emph{ is a geometrically integral canonical del Pezzo surface}\right\}, \\
\mathcal{X}_{\textup{dP}, \varepsilon}^{\textup{tame}} &=& \left\{ X \mid X \emph{ is a geometrically integral tame } \varepsilon\emph{-klt del Pezzo surface} \right\}, \text{ and} \\
\mathcal{X}_{\textup{dP}, \varepsilon}^{>5} &=&  \left\{ X \mid X \emph{ is an } \varepsilon\emph{-klt del Pezzo surface s.t. char}(H^0(X, \mathcal{O}_X)) \neq 2, 3, 5 \right\}.
\end{eqnarray*}

\end{theorem}

We briefly explain the organisation of the article. In \autoref{s-preliminaries}, we collect various results on geometry over imperfect fields and del Pezzo surfaces.
In \autoref{s-boundedness}, we generalise the main results of Tanaka \cite{Tan19} to the canonical case. We use Ekedahl's technique \cite{Eke88} on the construction of $\alpha$-torsor to show an effective Kodaira vanishing theorem (\autoref{p-bound-b&b}) from which we deduce that $\omega_X^{-12}$ is very ample (\autoref{c-12-v-ample}). 
Starting from \autoref{s-irregularity} we specialise to the study of geometrically integral del Pezzo surfaces. 
We show that the Frobenius length of geometric non-normality (an invariant introduced by Tanaka \cite{Tan21}) is at most 1 (\autoref{c-length-frob}) on normal Gorenstein del Pezzo surfaces, a result we use to find lower bounds on the dimension of the space of anti-pluricanonical sections.
We combine these estimates together with Maddock's bound \cite[Corollary 1.2.6]{Mad16} and a careful study of $\alpha$-torsors to prove \autoref{t-intro-irr}.
In \autoref{s-BAB} we apply our results to the BAB conjecture over arbitrary fields and we prove \autoref{t-intro-bdd}.
\\ 

\textbf{Acknowledgements.}
We would like to thank A.~Fanelli, S.~Filipazzi, J.~Waldron, and H.~Tanaka for useful discussions on the topic of this article.
This work started at EPFL and then continued at the University of Bonn during mutual visits. We would like to thank both universities for the support.
Part of this work was written while the first author was visiting  l'Université de Poitiers within the program “Chaire Ali\'{e}nor" of the Fédération MARGAUx: he would like to thank S. Boissière for the kind hospitality.
FB was partly supported by the grant $\#200021/169639$ from the
Swiss National Science Foundation.

\section{Preliminaries} \label{s-preliminaries}

\subsection{Notations} \label{sec: terminology}

\begin{enumerate}
    \item Given a field $k$, we denote by $\overline{k}$ (resp. $k^{\text{sep}}$) an algebraic (resp. separable) closure. We denote by $k^{1/p^{\infty}}$ the perfect closure of $k$.
    \item Given a field $k$, a scheme $X$ is a $k$-variety if it is an integral separated scheme of finite type over $k$.
    If $X$ has dimension 1 (resp. 2, 3), we say $X$ is a curve (resp. surface, 3-fold).
    \item Given a projective integral $k$-variety $X$, we let $d_X:=[H^0(X, \mathcal{O}_X):k]$.
    \item Given an $\mathbb{F}_p$-scheme $X$, we denote by $F \colon X \to X$ the absolute Frobenius morphism of $X$. 
    We say $X$ is $F$-finite if $F$ is a finite morphism.
    \item For an $F$-finite field $k$, its $p$-degree (or degree of imperfection) is defined as $\pdeg(k):=\log_{p}[k:k^p]$.
    \item  We say $(X, \Delta)$ is a \emph{pair} if $X$ is a normal $k$-variety, $\Delta$ is an effective $\mathbb{Q}$-divisor with coefficients in $[0, 1]$ and $K_X+\Delta$ is a $\mathbb{Q}$-Cartier divisor.
    \item For the definitions of the singularities of the MMP (as canonical, klt and log canonical), we refer to \cite[Definition 2.8]{kk-singbook}. 
    \item \label{item: terminology_conductor} Given an integral scheme $X$ with normalisation $\nu \colon Y \to X$, we denote by $\mathcal{I} \subset \mathcal{O}_X$ the conductor ideal (i.e. the annihilator of the $\mathcal{O}_X$-module $\nu_*(\mathcal{O}_Y)/\mathcal{O}_X)$).
    The corresponding closed subscheme $D \subset X$ is called the \emph{conductor scheme} of $\nu$.
    Note that $\mathcal{I}$ is also an ideal of $\mathcal{O}_Y$ and the corresponding subscheme $C \subset Y$ is called \emph{ramification locus} of $\nu$. 
    \item A projective morphism $f \colon X \to Y$ of normal schemes is a \emph{contraction} if $f_*\mathcal{O}_X=\mathcal{O}_Y$.
\end{enumerate}

\subsection{Geometric reducedness and normality}

We collect well-known results on the geometry of algebraic varieties, especially surfaces, defined over imperfect fields that we need in this article. 

\begin{definition}
    A $k$-variety $X$ is \emph{geometrically reduced} (resp. \emph{geometrically normal, geometrically regular}) if the base change $X_{\overline{k}}$ is reduced (resp. normal, regular).
\end{definition}

We recall Tate's base change formula for purely inseparable field extensions.

\begin{theorem}[{\cite[Theorem 1.1]{PW}}]
    Let $X$ be a normal $k$-variety such that $k$ is algebraically closed in $K(X)$. 
    Let $Y$ be the normalisation of the reduced scheme $(X \times_k \overline{k})_{\red}$ together with the natural morphism $f \colon Y \to X$. 
    Then there exists an effective divisor $C \geq 0$ such that 
    $K_Y+(p-1)C=f^*K_X.$
    If $X$ is geometrically integral, then $(p-1)C$ can be chosen to be the ramification divisor of $f$.
\end{theorem}

We start with the behaviour of geometric reducedeness under birational equivalence.

\begin{lemma}[{\cite[Lemma 2.2]{BT22}}] \label{l-geom-red}
Let $X$ and $Y$ be two $k$-birational varieties. Then $X$ is geometrically reduced over $k$ if
and only if $Y$ is geometrically reduced over $k$.
\end{lemma}

Next, we note that geometric normality descends under birational contractions.
For the definition of the $(S_n)$-property we refer to \cite[\href{https://stacks.math.columbia.edu/tag/033Q}{Tag 033Q}]{stacks-project}.

\begin{proposition}\label{p-descend-geom-normal}
Let $\pi \colon X \to Y$ be a projective birational morphism of normal $k$-varieties. If $X$ is geometrically normal, so is $Y$.
\end{proposition}

\begin{proof}
Recall that a variety $X$ over $k$ has the property $(S_n)$ if and only if $X_{\overline{k}}$ also has, by faithfully flat descent. 
As $Y$ is $(S_2)$, by Serre's criterion \cite[\href{https://stacks.math.columbia.edu/tag/031S}{Tag 031S}]{stacks-project} $Y$ is geometrically normal if and only if it is geometrically $(R_1)$.
Suppose by contradiction that there exists a codimension 1 point $\eta \in Y$ such that the localisation $\mathcal{O}_{Y, \eta}$ is not geometrically regular. 
As $Y$ is normal, $\pi$ is an isomorphism over codimension 1 points of $Y$ and thus $X$ is not geometrically $(R_1)$, reaching the contradiction.
\end{proof}

We discuss singularities of the MMP over imperfect fields.

\begin{definition} \label{def: geometrically-canonical}
 Let $(X,\Delta)$ be a pair over $k$ such that $k$ is algebraically closed in $K(X)$.
 We say it is \emph{geometrically canonical} (resp.~ klt, log canonical), if the base change $(X_{\overline{k}}, \Delta_{\overline{k}})$ is so.
\end{definition}

In particular, note that geometrically log canonical implies geometrically normal. If $X$ is geometrically canonical (resp. klt, lc), then $X$ is also canonical (resp. klt, lc) by \cite[Proposition 2.3]{BT22}.
We now specialise to the case of surfaces. Recall that the existence of resolution of singularities for excellent surfaces has been proven in \cite{Lip78}.

\begin{proposition}\label{p-klt-rational}
 Let $X$ be the spectrum of a local excellent ring $(R,\mathfrak{m})$ with closed point $x$.
If $(X, \Delta)$ is a klt surface pair for some $\Delta \geq 0$, then $X$ has rational and $\mathbb{Q}$-factorial singularities.
Therefore, if two projective $k$-surfaces $X$ and $Y$ with klt singularies are $k$-birational, then $H^i(X, \mathcal{O}_X) \simeq H^i(Y, \mathcal{O}_Y)$ for every $i \geq 0$.
\end{proposition}
\begin{proof}
Rationality of klt surface singularities follows from \cite[Proposition 2.28]{kk-singbook} and $\mathbb{Q}$-factoriality of rational singularities is proven in \cite[Proposition 17.1]{Lip69}. 
The last statement is obvious by considering a common resolution of $X$ and $Y$.
\end{proof}

\begin{corollary}\label{cor: surface_canonical_rational}
    Let $(x \in X)$ be a Gorenstein normal surface singularity.
    Then $X$ is canonical if and only if it is rational.
\end{corollary}

\begin{proof}
If $X$ is canonical, then it is rational by \autoref{p-klt-rational}.
Suppose now that $X$ is rational and let $f \colon Y \to X$ be a resolution of singularities. 
As $X$ is Gorenstein and $X$ has rational singularities, we have that $f_*\omega_Y = \omega_X$ by \cite[Proposition 2.77]{kk-singbook}, which in turn implies that $X$ has canonical singularities by \cite[Claim 2.3.1]{kk-singbook}.
\end{proof}

\subsection{Del Pezzo surfaces}

In this subsection, we collect some terminology on del Pezzo surfaces and recall previously known results. 

\begin{definition}
We say $X$ is a \emph{Gorenstein} (resp. \emph{canonical}, \emph{regular}) \emph{del Pezzo surface} over $k$ if $X$ is a reduced $k$-projective Gorenstein (resp.~ canonical, regular) surface with $H^0(X, \mathcal{O}_X)=k$ and $\omega_X^{-1}$ is ample.
We say $X$ is a \emph{weak} del Pezzo if $\omega_X^{-1}$ is big and nef.
\end{definition}

We recall the classification of Gorenstein normal del Pezzo surfaces over algebraically closed fields:

\begin{proposition}[{\cite[Theorem 2.2]{HW81}}] \label{l-HW}
Let $X$ be a normal Gorenstein del Pezzo surface over an algebraically closed field $k$. 
Then one of the following holds:
\begin{enumerate}
    \item $X$ is a canonical del Pezzo surface and the explicit list is described in \cite[Section 8]{Dol12}, or %\cite[Theorem 3.4]{HW81}
    \item the minimal resolution $Z \to X$ is a ruled surface of the form $\mathbb{P}_E(\mathcal{O}_E \oplus \mathcal{L})$, where $E$ is an elliptic curve and $\deg \mathcal{L}<0$. The surface $X$ is obtained by contracting the negative section of $Z$.
\end{enumerate}
\end{proposition}

In \cite[Theorem 3.3]{BT22}, it is shown that canonical del Pezzo surfaces which are geometrically normal are geometrically canonical. 
We present a different proof of this result relying on \autoref{l-HW} and the following observation: 

\begin{lemma} \label{l-lc-rational}
Let $(y \in Y)$ be a geometrically log canonical surface singularity over $k$.
Suppose that $Y$ has rational singularities.
Then $Y_{\overline{k}}$ has rational singularities.
\end{lemma}
\begin{proof}
We can suppose $k$ is separably closed and $Y$ is the spectrum of a local henselian ring $(R, \mathfrak{m})$  by the existence of resolution of singularities \cite{Lip78}.
Let $U \coloneqq \Spec(R) \setminus \left\{\mathfrak{m} \right\}$ be the punctured spectrum. 
Since $Y$ is rational, the group $\Pic(U)$ is finite by \cite[Proposition 17.1]{Lip69}.
Therefore also $X:=Y_{\overline{k}}$ is $\mathbb{Q}$-factorial by \cite[Lemma 2.5]{Tan18a} and thus $\Pic(U_{\overline{k}})$ is a torsion group. 
Let $f \colon W \to X$ be the minimal resolution with exceptional divisor $E=\sum_{i=1}^n E_i$.
As defined in \cite{Lip69}, $\Pic^0(W)$ is the group of line bundles $L$ on $W$ such that $L \cdot E_i=0$ for every $i$ and there is an exact sequence of groups $0 \to \Pic^0(W) \to \Pic(W) \to \bigoplus \mathbb{Z}[E_i] \to 0$. 
By \cite[Proposition 14.4]{Lip69}, $\Pic^0(W)$ embeds into $\Pic(U_{\overline{k}})$ and thus we deduce it is a torsion group.
%we have the following commutative diagram
%\begin{equation}
%\begin{tikzcd}
%    0 \ar[r] & \Pic^0(W) \ar[r] \ar[d] & \Pic(W) \ar[r] \ar[d] & \mathbb{Z}^n \ar[r] \ar[d] & 0 \\
%    0 \ar[r] & \Pic^0(E) \ar[r] & \Pic(E) \ar[r] & \mathbb{Z}^n  \ar[r] & 0, \\ 
%\end{tikzcd}

%\end{equation}

Suppose now by contradiction that $X$ is not rational. By the classification of log canonical singularities \cite[Corollary 3.39]{kk-singbook}, the exceptional divisor $E$ is either an elliptic curve, a nodal curve or a circle of smooth rational curves. 
In the first case $\Pic^0(E)\simeq E(k)$, while in the latter cases $\Pic^0(E) \simeq k^*$ by \cite[Chapter 9.3, Corollary 11 and 12]{BLR} and since $h^1(E,\mathcal{O}_E) = 1$. By \cite[Lemma 14.3]{Lip69}, the restriction map $\Pic(W) \to \Pic(E)$ is surjective. 
Considering the exact sequence $0 \to \Pic^0(E) \to \Pic(E) \to \mathbb{Z}^{n} \to 0$, we can deduce that the map $\Pic^0(W) \to \Pic^0(E)
$ is surjective.
This is a contradiction as $\Pic^0(W)$ is torsion while $k^*$ and $E(k)$ are not.
\end{proof}

\begin{proposition} \label{prop: geom-normal-canonical-dP}
Let $X$ be a canonical del Pezzo surface. 
If $X$ is geometrically normal, then it is geometrically canonical.
\end{proposition}

\begin{proof}
By \autoref{l-HW}, $X$ is geometrically log canonical. 
As $X$ has rational singularities, $X$ is geometrically rational by \autoref{l-lc-rational}. As $X$ is Gorenstein, we conclude that $X$ is geometrically canonical by \autoref{cor: surface_canonical_rational}.
\end{proof}

We now recall the results of Reid on the classification of non-normal Gorenstein del Pezzo surfaces \cite{Rei94}. 
We fix some notations we will use throughout the article (cf. \autoref{sec: terminology} for the terminology used).

\begin{definition}\label{not-non-norm-dp}
Let $X$ be a non-normal integral Gorenstein del Pezzo surface with normalisation $\nu \colon Y \to X$.
We say $X$ is \emph{tame} if $H^1(X, \mathcal{O}_X)=0$.
\end{definition}

One can characterise tame del Pezzo surfaces in terms of the conductor. 

\begin{theorem} \label{t-recap-reid}
Let $X$ be a non-normal integral Gorenstein del Pezzo surface over an algebraically closed field. Then the conductor $D \subset X$ is integral. Moreover, : 
\begin{enumerate}
    \item \label{item: tame} $X$ is tame if and only if $D \simeq \mathbb{P}^1$;
    \item \label{item: divisibility} $(p-1)$ divides $h^1( \mathcal{O}_X)$.
\end{enumerate}
\end{theorem}

\begin{proof}
    The integrality of the conductor follows from \cite[Lemma, page 718]{Rei94} for integral del Pezzo surfaces.
    Then, \eqref{item: tame} follows from  the proof of \cite[Corollary 4.10]{Rei94}, as $D$ is irreducible.
    \eqref{item: divisibility} is proved in \cite[4.11]{Rei94}
\end{proof}

We will repeatedly use the following:

\begin{lemma}\label{l-image}
Let $\pi \colon X \to Y$ be a proper birational morphism of $k$-surfaces.
If $X$ is a regular (resp. canonical) del Pezzo surface, then so is $Y$.
If $X$ is a regular (or canonical) weak del Pezzo surface, then also $Y$ is a canonical weak del Pezzo surface.
\end{lemma}

\begin{proof}
We only prove the case where $X$ is a regular weak del Pezzo surface, as the others are similar.
As $-K_X$ is $\pi$-big and $\pi$-nef, we conclude that $Y$ has canonical singularities by the negativity lemma \cite[Lemma 2.11]{Tan18b}. As $-K_Y=\pi_*(-K_X)$ we conclude by projection formula that $-K_Y$ is big and nef.
\end{proof}

From the point of view of the MMP, it is natural to consider surfaces of del Pezzo type. For their basic properties we refer to \cite[Section 2.3]{BT22}.

\begin{definition} \label{def: dPtype}
We say $X$ is a \emph{surface of del Pezzo type} over $k$ if $X$ is a projective $k$-variety with $H^0(X, \mathcal{O}_X)=k$ and there exists $\Delta \geq 0$ such that $(X,\Delta)$ is a log del Pezzo pair (i.e. $(X,\Delta)$ klt and $-(K_X+\Delta)$ is big and nef).
\end{definition}

The following describes the Picard scheme of del Pezzo surfaces.

\begin{proposition} \label{prop: picard-dP}
    Let $X$ be a surface of del Pezzo type. 
    Then $\Pic^0_{X/k}$ is a unipotent smooth commutative $k$-group scheme of finite type over $k$ of dimension $h^1(X, \mathcal{O}_X)$.
\end{proposition}

\begin{proof}
    By Serre duality, we have $H^2(X, \mathcal{O}_X)=H^0(X, \omega_X)=0$ and therefore by \cite[Corollary 9.4.18.3, Corollary 9.5.13 and Remark 9.5.15]{FAG} the group scheme $\Pic^0_{X/k}$ is smooth of dimension $h^1(\mathcal{O}_X)$.
    We are left to show that $\Pic^0_{X/k}$ is unipotent.
    For this, we can suppose $k$ is separably closed.
    By \cite[Theorem 1.3]{BT22}, there exists $n>0$ such that for every $L \in \Pic^0(X)$ we have $L^{\otimes p^n} \simeq \mathcal{O}_X$.
    In other words, multiplication by $p^n$ on $\Pic^0_{X/k}$ coincides with the zero homomorphism on $k$-rational points.
    By density of rational points \cite[Proposition 3.5.70]{Poo17} and since $\Pic^0_{X/k}$ is reduced, we conclude that taking $p^n$-powers on $\Pic^0_{X/k}$ coincides with the zero homomorphism as a morphism of schemes and thus $\Pic^0_{X/k}$ is unipotent.
\end{proof}

\section{Bounds on the anticanonical volume and effective very ampleness}\label{s-boundedness}

In this section we prove bounds on the anticanonical volume and very ampleness statements for canonical del Pezzo surfaces over imperfect fields.

\subsection{Bounding volumes}
We start by bounding the volume of canonical del Pezzo surfaces in terms of their thickening exponent $\epsilon(X/k)$ (see \cite[Definition 7.4]{Tan19} and \autoref{def: thickeningexponent} below).
First, we need an explicit bound on the Cartier index of a klt surface singularity.
For a $\mathbb{Q}$-factorial variety $X$, we define its \emph{Cartier index} to be the smallest integer $n>0$ such that for every Weil divisor $D$ in $X$, the Weil divisor $nD$ is Cartier.

\begin{lemma}\label{l-Cartier-index}
Let $X$ be the spectrum of a local $k$-algebra $(R,\mathfrak{m})$, and let $x$ be the closed point corresponding to $\mathfrak{m}$.
Suppose $(X, \Delta)$ is a klt surface pair for some $\Delta \geq 0$.
Let $f \colon Y \to X$ be the minimal resolution of singularities, with exceptional divisor $E=\sum_{i=1}^n E_i$.
Let $M=(E_i \cdot_k E_j)_{i,j=1}^n$ be the intersection matrix and let $d= \det(M)$. Then there exists $d_x$ such that $d=d_x[k(x):k]$ and the Cartier index of $X$ divides $d_x$.
\end{lemma}

\begin{proof}
Recall that $X$ is rational and $\mathbb{Q}$-factorial by \autoref{p-klt-rational}.
Let $D$ be a Weil integral divisor on $X$ and write $f^{*} D = f_*^{-1} D + \sum_{i=1}^n a_i E_i$ for some $a_i \in \mathbb{Q}$. 

We claim it is sufficient to show $d_xa_i$ is integral. Indeed, then $f^{*} (d_xD)$ is an integral divisor on a regular surface and thus $f^{*} (d_xD)$ is Cartier. 
If we write $K_Y+\Delta_Y=f^*(K_X+\Delta)$, then $\Delta_Y$ is effective by the negativity lemma and $(Y,\Delta_Y)$ is klt.
As $f^*(d_xD)-(K_Y+\Delta_Y)$ is $f$-nef and big,
and $f^*(d_xD)$ is $f$-trivial, there exists $b_0>0$ such that for all $b \geq b_0$ we have that $bf^*(d_xD)=f^*A_{b}$ for a Cartier divisor $A_b$ on $X$ by the base point free theorem for excellent surfaces \cite[Theorem 4.4]{Tan18b}.
Then $f^*(d_xD)=(b_0+1)f^*d_xD-b_0f^*d_xD=f^*(A_{b_0+1}-A_{b_0})$ and thus $d_xD$ is Cartier.

We denote by $(a_i)$ (resp. $f_*^{-1} D \cdot E_j$) the vector $(a_1, \dots, a_n)$ (resp. $(f_*^{-1} D \cdot E_1, \dots, f_*^{-1} D \cdot E_n)$). Given a closed point $x \in X$, we denote by $k(x)$ the residue field of $X$ at $x$.
By the projection formula, $$(a_i)=M^{-1} (-f_*^{-1} D \cdot E_j)=\frac{1}{d_x [k(x):k]} A (f_*^{-1} D \cdot E_j),$$ where $A$ is a matrix with integer coefficients. 
We have $(-f_*^{-1} D \cdot E_j) = \sum_j m_j  [k(y_j):k]$ for some $m_j \in \mathbb{Z}$, where the $y_j$ are the intersection points of $f_*^{-1}D$ with $E_j$. As $k(x) \subset k(y_j)$, we conclude that $[k(x):k]$ divides $(f_*^{-1} D \cdot E_j)$, thus showing $d_x a_i$ is an integer.
\end{proof} 

We bound the volume of canonical del Pezzo surfaces, generalising the regular case proven in \cite[Theorem 4.7]{Tan19}. 

\begin{definition}[{\cite[Definition 5.1, Definition 7.4]{Tan21}}] \label{def: thickeningexponent}
Let $X$ be a
normal variety over $k$ such that $k$ is algebraically closed in $K(X)$. We
define the \emph{Frobenius length of geometric non-normality} $\ell_F(X/k)$ as 
$$ \ell_F(X/k):= \text{min} \left\{ e \geq 0 \mid (X \times_k k^{1/p^e})^{\norm}_{\text{red}} \text{ is geometrically normal over } k^{1/p^e} \right\}.
$$
Set $R$ to be the
local ring of $X \times_k k^{1/p^{\infty}}$
at the generic point. 
We define the \emph{thickening exponent} $\epsilon(X/k)$ as the non-negative integer such that
$\length_R R = p^{\epsilon(X/k)}$ 
\end{definition}

For a discussion of the properties of $\ell_F(X/k)$ and $\epsilon(X/k)$, we refer the reader to \cite[Section 5, Section 7]{Tan21}.

We fix some notation.
For $d \geq 1$, we denote the Hirzebruch surface $\mathbb{P}_{\mathbb{P}^1}(\mathcal{O}_{\mathbb{P}^1} \oplus \mathcal{O}_{\mathbb{P}^1}(-d))$ by $\mathbb{F}_d$, a closed rational fibre by $F$ and the negative section by $C_d$.
The contraction of $C_d$ is the morphism $p \colon \mathbb{F}_d \to \mathbb{P}(1,1,d)$ and we denote by $L \coloneqq p_*F$ the generator of its class group. Recall that $L \in |\mathcal{O}_{\mathbb{P}(1,1,d)}(1)|$, and that $L^2 =\frac{1}{d}$. 

\begin{lemma}\label{lem: Cartier_index_weigthed}
    The divisor class $nK_{\mathbb{P}(1,1,d)}$ is Cartier if and only if $d \mid n(d+2).$
\end{lemma}

\begin{proof}
    As $K_{\mathbb{P}(1,1,d)} \sim (-d-2)L$ and the Cartier index of $L$ is $d$, the lemma is immediate.
\end{proof}

\begin{proposition}\label{p-bound-volume-canonical-dP}
Let $X$ be a canonical del Pezzo surface. Then
\begin{enumerate}
    \item if $X$ is geometrically normal, then it is geometrically canonical and $K_X^2 \leq 9$;
    \item if $X$ is not geometrically normal, then $p\in \left\{2,3 \right\}$ and 
    \begin{enumerate}
        \item if $p=3$, $\ell_F(X/k)=1$ and $K_X^2 \leq 12 \cdot 3^{\epsilon(X/k)}$.
        \item if $p=2$, $\ell_F(X/k)\leq 2$ and $K_X^2 \leq 16 \cdot 2^{\epsilon(X/k)}$.
    \end{enumerate}
\end{enumerate}
\end{proposition}

\begin{proof}
We can assume $k$ to be separably closed and we will repeatedly use the fact that $\epsilon(X/k)$ is a $k$-birational invariant \cite[Proposition 7.10]{Tan21}.
If $X$ is geometrically normal, then we conclude by \autoref{prop: geom-normal-canonical-dP}. 
So we suppose that $X$ is not geometrically normal and $p =2,3$ by \cite[Theorem 3.7.(1)]{BT22} .
The bounds on $\ell_F(X/k)$ are proven in \cite[Theorem 3.7.(2)-(3)]{BT22}.

Let $Z \to X$ be the minimal resolution of $X$. 
As $Z$ is a regular weak del Pezzo surface, by \autoref{l-image} a $K_Z$-MMP will end with a regular weak del Pezzo surface $Y$ admitting a Mori fibre space $f \colon Y \to B$, i.e. $f$ is a contraction where $-K_Y$ is $f$-ample and $\dim(B) \leq 1$.
Note that $K_Y^2 \geq K_Z^2=K_X^2$.

If $B=\Spec(k)$, then $Y$ is a regular del Pezzo surface and we conclude by \cite[Theorem 4.7]{Tan19}.
If $B$ is a curve, as $Y$ is weak del Pezzo, the cone theorem \cite[Theorem 2.14]{Tan18b} implies that the Mori cone of $Y$ is $$\NE(Y)=\mathbb{R}_{+}[F]+\mathbb{R}_{+}[\Gamma],$$ 
where $F$ the class of a closed fibre of $f$ and $\Gamma$ is the class of an integral curve with self-intersection $\Gamma^2 \leq 0$.
If $K_Y \cdot_k \Gamma<0$, then $Y$ is a regular del Pezzo surface by Kleiman's criterion and we conclude again by \cite[Theorem 4.7]{Tan19}.

If $K_Y \cdot_k \Gamma=0$, by the Hodge index theorem $\Gamma^2<0$ and, if we denote $k_{\Gamma}=H^0(\Gamma, \mathcal{O}_{\Gamma})$, by adjunction the equality $\Gamma^2=\deg_k \omega_{\Gamma/k}=-2[k_{\Gamma}:k]$ holds. Then there exists a birational contraction $Y \to T$ where $T$ is a canonical del Pezzo surface of Picard rank 1 with a unique singular point $x$ and $K_T^2=K_Y^2$. 
As $k_\Gamma=k(x)$ by \cite[Corollary 10.10]{kk-singbook}, we have $\Gamma^2=-2[k(x):k]$, which implies that the Cartier index of $T$ divides $2$ by \autoref{l-Cartier-index}.
If $T$ is geometrically normal, it is geometrically canonical by \cite[Theorem 3.7]{BT22}. Moreover, 
as $T_{\overline{k}}$ has Picard rank 1 and a singular point, we conclude $K_X^{2} \leq 8$.
If $T$ is not geometrically normal and $g \colon V=(T \times_k \overline{k})_{\red}^{\norm} \to T$ is the normalised base change where $K_V+(p-1)C=g^*K_T$ we deduce that $2p^{\ell_F(T/k)} K_V$ is Cartier by \cite[Theorem 5.12]{Tan21}.
By the classification of the normalised base changes of canonical del Pezzo surfaces with Picard rank 1 \cite[Theorem 4.1]{PW}, the bounds on the Frobenius length \cite[Theorem 3.7]{BT22} and \autoref{lem: Cartier_index_weigthed} we deduce the following:
\begin{itemize}
    \item\label{Case: p=3-volume} if $p=3$, then $6K_V$ is Cartier and thus $V \simeq \mathbb{P}(1,1,d)$ for $d \in \left\{1, 2, 3, 4, 6, 12\right\}$ and $C =L$;
    \item\label{Case: p=2-volume} if $p=2$, then $8K_V$ is Cartier and thus $V \simeq \mathbb{P}(1,1,d)$ for $d \in \left\{1, 2, 4, 8, 16\right\}$ and $C=L$ or $2L$ by \cite[Proposition 4.1]{Tan19}.
\end{itemize}
Using \cite[Lemma 4.5]{Tan19} we have 
$$p^{\epsilon(X/k)}(g^*K_T)^2=p^{\epsilon(X/k)} (K_V+(p-1)C)^2 = K_T^2.$$
If $p=3$, we have $V=\mathbb{P}(1,1,d)$, $C=L$ and thus 
$K_T^2 \leq (dL)^2 \cdot 3^{\epsilon(X/k)}= d \cdot 3^{\epsilon(X/k)} \leq 12 \cdot 3^{\epsilon(X/k)} $.
Similarly, in the case where $p=2$ we obtain that $K_T^2 \leq 16 \cdot 2^{\epsilon(X/k)}$.
\end{proof}

Using the bounds on the anticanonical volume, we can restrict the possibilities for the normalised base changes of non-normal canonical del Pezzo surfaces obtained in \cite[Theorem 4.1]{PW}. 
For the analogous result in the regular case, see \cite[Theorem 4.6]{Tan19}.

\begin{theorem}\label{t-recap}
Let $X$ be a canonical del Pezzo surface.
Let $\nu \colon Y \to (X \times_k \overline k)_{\red}$ be the normalisation morphism and let $f \colon Y \to X \times_k \overline k $
be the composite morphism. 
\begin{enumerate}
\item If $X$ is geometrically normal, then it is geometrically canonical. 
\item If $p \geq 5$, then $X$ is geometrically normal.
\item If $p=3$ and $X$ is not geometrically normal, then $\ell_{F}(X/k) = 1$  and $(Y, C)$ is isomorphic to $(\mathbb{P}(1,1,d), L)$ for some $d \leq 12$.
\item 
If $p=2$ and $X$ is not geometrically normal, then $\ell_{F}(X/k) \in \left\{1,2\right\}$ and $(Y,C)$ is isomorphic to one of the following:
\begin{enumerate}
    \item $(\mathbb{P}^2,L)$ and $\ell_F(X/k)=1$;
    \item $(\mathbb{P}^2, C \in |2L|)$;
    \item \label{item:weighted} $(\mathbb{P}(1,1,d),2L)$ for $2 \leq d \leq 16$.
    \item $(\mathbb{P}^1 \times \mathbb{P}^1, C \in |F_1+F_2|)$ and $\ell_F(X/k)=1$;
    \item $(\mathbb{P}^1 \times \mathbb{P}^1, F_i)$ and $\ell_F(X/k)=1$;
    \item \label{item: 1st_hirzebruch} $(\mathbb{F}_d, D \in |C_d+F|)$, where $C_d$ is the negative section and $\ell_F(X/k)=1$ for $1 \leq d \leq 14$;
    \item \label{item: 2nd_hirzebruch} $(\mathbb{F}_d, C_d)$ and $\ell_F(X/k)=1$ for $1 \leq d \leq 12$;
    \end{enumerate}
\end{enumerate}
\end{theorem}

\begin{proof}
By \autoref{p-bound-volume-canonical-dP}, we are only left to prove the classification in (3) and (4).
Suppose $p=3$. 
The only possible normalised base change is $\mathbb{P}(1,1,d)$ by \cite[Proposition 4.1 and Remark 4.3]{Tan19}.
However, by \autoref{p-bound-volume-canonical-dP}, we have $K_X^2=p^{\epsilon(X/k)} d \leq 12 \cdot p^{\epsilon(X/k)} $.

Suppose $p=2$. The list of possibilities without the bounds on $d$ is proved in \cite[Proposition 4.1]{Tan19}.
It is now sufficient to note that in Case \eqref{item: 1st_hirzebruch} $K_X^2=p^{\epsilon(X/k)}(d+2)$, in Case \eqref{item: 2nd_hirzebruch} $K_X^2=p^{\epsilon(X/k)}(d+4)$, and in Case \eqref{item:weighted} $K_X^2=p^{\epsilon(X/k)}d$.
Using \autoref{p-bound-volume-canonical-dP} we deduce the desired bounds on $d$.
\end{proof}

\subsection{Effective Kodaira vanishing and very ampleness on del Pezzo surfaces}

In this section we prove an effective version of the Kawamata--Viehweg vanishing theorem on canonical del Pezzo surfaces.
From this we deduce bounds on the effective global generation and very ampleness for the anti-pluricanonical linear systems. 

We start by giving an effective version of \cite[Theorem 1.9]{PW} in the 2-dimensional case.

\begin{proposition}\label{p-bound-b&b}
Let $X$ be a canonical del Pezzo surface and let $A$ be a big and nef Cartier divisor on $X$. Then
\begin{enumerate}
    \item if $p>3$, then $H^1(X, \mathcal{O}_X(-A))=0$;
    \item if $p=3$, then $H^1(X, \mathcal{O}_X(-dA))=0$ if $d \geq 2$;
    \item if $p=2$, then $H^1(X, \mathcal{O}_X(-dA))=0$ if $d \geq 4$.
\end{enumerate}
If $X$ is a normal Gorenstein del Pezzo surface, the same results hold if $A$ is ample.
\end{proposition}

\begin{proof}
We let $\mathcal{A}_m=\mathcal{O}_X(mA)$ for $m \in \mathbb{Z}$. 
We fix $d>0$. 
We show that  $H^1(X, \mathcal{A}_{-dn})=0$ for $n$ sufficiently large. If $A$ is ample, we conclude by Serre duality and Serre vanishing. 
If $A$ is only big and nef and $X$ is a canonical del Pezzo surface, by the base point free theorem \cite[Proposition 2.1]{Ber21} there is a birational contraction $\pi \colon X \to Y$ such that $A=\pi^*H$, where $H$ is an ample Cartier divisor and by \autoref{l-image} $Y$ is a del Pezzo surface with canonical singularities. Thus the singularities of $Y$ are rational by \autoref{p-klt-rational} and the projection formula implies 
$H^1(X, \mathcal{A}_{-dn})=H^1(Y, \mathcal{H}_{-dn})$.
As $Y$ is a normal, we can apply Serre duality to deduce $H^1(Y, \mathcal{H}_{-dn}) \simeq H^1(Y, \omega_X \otimes \mathcal{H}_{dn})$, which vanishes for $n$ large enough by Serre vanishing. 

Suppose $H^1(X,\mathcal{A}_{-d}) \neq 0$.
 Without loss of generality by the previous paragraph, we can assume $F^*\colon H^1(X,\mathcal{A}_{-d}) \to H^1(X,\mathcal{A}_{-pd})$ has a non-trivial element $\zeta$ in the kernel $H^1_{\fppf}(X, \alpha_{\mathcal{A}_{-d}})$.  
By \cite[Theorem 2.11]{PW}, associated to $\zeta$ there exists a degree $p$ purely inseparable morphism $\pi \colon Z \to X$ such that $Z$ is an integral Gorenstein surface with $\omega_Z=\pi^*(\omega_X(-(p-1)dA))$.
Let $\nu \colon Z^{\norm} \to Z$ be the normalisation and let $\mu \colon Y:=(Z^{\norm} \times_k \overline{k})_{\red}^{\norm} \to Z^{\norm}$ be the normalised base change to the algebraic closure. 
We denote by $\Gamma$ the divisorial part of the ramification locus. We have $\mathcal{O}_{Z^{\norm}}(K_{Z^{\norm}}+\Gamma)=\nu^*(\omega_Z)$ and there exists an effective Weil divisor $D \geq 0$ such that $K_Y+(p-1)D = \mu^*K_{Z^{\norm}}$ by \cite[Theorem 1.1]{PW}, we conclude
 $$K_Y+(p-1)D+\mu^*\Gamma= f^*(K_X-(p-1)dA),$$
 where $f =\pi \circ \nu \circ \mu$.
 Consider a general curve $C$ on $Y$ of genus $g \geq 1$ so that $C$ is contained in the smooth locus of $Y$ and $C \cdot ((p-1)D+\mu^*\Gamma) \geq 0$.
 Therefore $K_Y \cdot C<0$, and the bend and break lemma \cite[Chapter II, Theorem 5.8]{Kol96} shows that for every point $x \in C$ there exists a rational curve $L_x$ such that
 $$-(K_Y+(p-1)D+\mu^*\Gamma) \cdot L_x \leq 4 \frac{-(K_Y+(p-1)D+\mu^*\Gamma)\cdot C}{-K_Y \cdot C} \leq 4, $$
 as $-(K_Y+(p-1)D+\mu^*\Gamma)$ is big and nef. 
 Since $-K_X$ is ample, we infer the inequality:
 $$ f^*((p-1)dA) \cdot L_x < f^*(-K_X+(p-1)dA) \cdot L_x \leq 4. $$
As $A=\pi^*H$ where $H$ is an ample Cartier divisor and $x$ is a general point on $C$, we have $f^*A \cdot L_x \geq 1$ and thus we have $(p-1)d \leq 3$, which concludes the proof.
\end{proof}

\begin{lemma}\label{l-gg-dp}
Let $X$ be a canonical del Pezzo surface such that $X$ is not geometrically normal. 
Let $A$ be a big and nef Cartier divisor on $X$. Then
\begin{enumerate}
    \item if $p=3$, then $\mathcal{O}_X(3A)$ is globally generated;
    \item if $p=2$ and $\ell_F(X/k)=1$, then $\mathcal{O}_X(2A)$ is globally generated;
    \item if $p=2$ and $\ell_F(X/k)=2$, then $\mathcal{O}_X(4A)$ is globally generated.
\end{enumerate}
\end{lemma}

\begin{proof}
The proof is the same as \cite[Theorem 3.5]{Tan19}.
There is a factorisation of the iterated Frobenius morphism by \cite[Theorem 5.9]{Tan21}:
$$F^{\ell_F(X/k)}_{X \times_k \overline{k}} \colon X \times_k \overline{k} \to (X \times_k \overline{k})_{\red}^{\norm} \xrightarrow{\mu} X \times_k \overline{k},$$
where $(X \times_k \overline{k})_{\red}^{\norm}$ is a toric variety by \autoref{t-recap}. Thus $\mu^*A$ is globally generated and also $(F^\ell_F(X/k)_{X \times_k \overline{k}})^*A=A^{p^{\ell_F(X/k)}}.$
\end{proof}

We recall the following very ampleness criterion for line bundles.
For the notion of Castelnuovo--Mumford regularity and its basic properties we refer to \cite[Section 1.8]{Laz04}.

\begin{proposition}[{\cite[Lemma 11.2]{Tan21}}]\label{p-very-ampleness-0}
Let $X$ be a geometrically irreducible $k$-projective variety of dimension $n$.
Let $\mathcal{A}$ be a globally generated ample line bundle and suppose $\mathcal{L}$ is an ample line bundle which is $0$-regular with respect to $\mathcal{A}$.
Then $\mathcal{A} \otimes \mathcal{L}$ is very ample.
\end{proposition}

The following is a generalisation of \cite[Theorem 3.5]{Tan19} including the case of canonical del Pezzo surfaces.

\begin{proposition} \label{t-very-ample}
Let $X$ be a canonical del Pezzo surface such that $X$ is not geometrically normal. Then
\begin{enumerate}
    \item if $p = 3$, then $\omega_X^{ -9}$ is very ample;
    \item if $p=2$ and $\ell_F(X/k)=1$, then $\omega_X^{ -7}$ is very ample;
    \item if $p=2$ and $\ell_F(X/k)=2$, then $\omega_X^{ -12}$ is very ample.
\end{enumerate}
\end{proposition}

\begin{proof}
By  \autoref{l-gg-dp} and \autoref{p-very-ampleness-0}  it is sufficient to verify that for $p=3$ (resp. $p=2, \ell_F(X/k)=1$ and $p=2, \ell_F(X/k)=2$) the line bundle $\omega_X^{ -6}$ (resp. $\omega_X^{ -5}$ and $\omega_X^{ -8}$)  is $0$-regular with respect to $\omega_X^{ -3}$ (resp. $\omega_X^{ -2}$ and $\omega_X^{ -4}$) to show the statement. 
We prove only the case $p=2$ and $\ell_F(X/k)=2$  as the others are analogous. 
In this case, $H^1(X, \omega_X^{ -8} \otimes \omega_X^{ 4})=H^1(X, \omega_X^{ -4})=H^1(X, \omega_X^{ 5})=0$ by \autoref{p-bound-b&b} and $H^2(X, \mathcal{O}_X)=H^0(X, \omega_X)=0$. 
\end{proof}

We now show the effective statements on very ampleness for the pluri-anticanonical systems.

\begin{theorem}\label{c-12-v-ample}
Let $X$ be a canonical del Pezzo surface. Then $\omega_X^{ -12}$ is very ample.
\end{theorem}
\begin{proof}
If $X$ is geometrically normal, then it is geometrically canonical by \autoref{prop: geom-normal-canonical-dP} and $\omega_X^{\otimes -6}$ is very ample by \cite[Proposition 2.14]{BT22}.
If $X$ is not geometrically normal, we apply \autoref{t-very-ample}.
\end{proof}

\section{Bounds on the irregularity} \label{s-irregularity}

In this section we study geometrically integral geometrically non-normal Gorenstein del Pezzo surfaces $X$.
The additional condition on geometric integrality allows to find additional constraints on the normalised base changes to the algebraic closure and the irregularity of $X$.

\subsection{A bound on $\gamma(X/k)$ for geometrically integral varieties}

Given a geometrically integral normal variety $X$ over $k$, we relate the $\delta$-invariant measuring the singularities in codimension 1 of $X_{\overline{k}}$ with the capacity of denormalising extensions $\gamma(X/k)$ introduced by Tanaka \cite[Section 4]{Tan21}.

\begin{definition}
For an integral $k$-variety $X$ with normalization $\nu \colon Y \to X$ with ramification $C \subseteq Y$ and conductor $D \subseteq X$, we define the \emph{$\delta$-invariant} of $X$ over $k$ as
$$
\delta(X/k) := \max_{\eta \in D} {\rm length}_{\mathcal{O}_{D,\eta}}(\mathcal{O}_{C, \eta} / \mathcal{O}_{D,\eta}),
$$
where $\eta$ runs over all generic points of irreducible components of $D$.
\end{definition}

\begin{proposition} \label{p-delta}
Let $X$ be a geometrically integral normal variety over a field $k$. 
Then
$$
\ell_F(X/k) \leq \gamma(X/k) \leq \delta(X_{\bar{k}}/\bar{k}).
$$
\end{proposition}

\begin{proof}
The inequality $\ell_F(X/k) \leq \gamma(X/k)$ is shown in \cite[Proposition 8.7]{Tan21}, so we are left to show $\gamma(X/k) \leq \delta(X_{\bar{k}}/\bar{k})$.
As the statement can be checked on an open covering of $X$, we can assume that the conductor $D$ of  $X_{\bar{k}}$ is irreducible, with generic point $\eta$.

By definition of $\gamma(X/k)$ \cite[Definition 4.1]{Tan21}, we can find a sequence of purely inseparable field extensions $k =: k_0 \subseteq k_1 \subseteq \hdots \subseteq k_{\gamma(X/k)}$ such that, if we inductively define $X_0 := X$ and $X_i := (X_{i-1,k_i})^{\norm}$, then $X_{i,{k_{i+1}}}$ is not normal and there is no longer sequence of fields with this property. In particular, $X_{\gamma(X/k)}$ is geometrically normal and the normalization $\nu: Y \to X_{\bar{k}}$ of $X_{\bar{k}}$ factors as
$$
\nu = \nu_1 \circ \cdots \circ \nu_{\gamma(X/k)}: Y = X_{\gamma(X/k),{\bar{k}}} \to \hdots \to X_{0,{\bar{k}}} = X_{\bar{k}}
$$
Note that each $X_{i,{\bar{k}}}$ has the property $(S_2)$, being the base change of a normal variety along a field extension. 

Now, after localizing at $\eta$, the factorization of $\nu$ corresponds to an ascending chain of subrings
$
\mathcal{O}_{X_{\bar{k}},\eta} = \mathcal{O}_{X_{0,{\bar{k}}},\eta} \subseteq \mathcal{O}_{X_{1,{\bar{k}}},\eta} \hdots \subseteq \mathcal{O}_{X_{\gamma(X/k),{\bar{k}}},\eta} = \mathcal{O}_{Y,\eta}.
$
%\FB{Sorry to be pedantic: $\eta$ lives on $X_\overline{k}$, not $Y$. what is the correct formulation?} \GM{I guess it should be $((\nu_1 \circ \hdots \circ \nu_i)_* \mathcal{O}_{X_i,{\bar k}})_{\eta}$, but that is a bit too heavy notation for my taste...}
Each inclusion $\mathcal{O}_{X_{i-1,{\bar{k}}},\eta} \subseteq \mathcal{O}_{X_{i,{\bar{k}}},\eta}$ is strict: otherwise $\nu_i$ would be an isomorphism in codimension $1$, hence so would be $X_i \to X_{i-1,k_i}$. Since $X_i$ is normal and $X_{i-1,k_i}$ has property $(S_2)$, this would imply that $X_{i-1,k_i}$ is normal as well, contradicting our choice of $k_i$.

By definition, % the conductor ideal $\mathcal{C}_{\eta} \subseteq \mathcal{O}_{X_{\bar k},\eta}$ is also an ideal in $\mathcal{O}_{Y,\eta}$. Hence, if we let $C$ be the preimage of $D$ in $Y$, then 
we have isomorphisms of $(\mathcal{O}_{X_{\bar k},\eta})$-modules 
$$
\mathcal{O}_{Y,\eta}/\mathcal{O}_{X_{\bar k},\eta} \cong (\mathcal{O}_{Y,\eta}/\mathcal{C}_{\eta})/(\mathcal{O}_{X_{\bar k},\eta}/\mathcal{C}_{\eta}) \cong \mathcal{O}_{C,\eta}/\mathcal{O}_{D,\eta}
$$
Note that both sides are annihilated by the conductor ideal $\mathcal{C}_{\eta}$, hence this is also an isomorphism of $(\mathcal{O}_{D,\eta})$-modules. Therefore, by strictness of $\mathcal{O}_{X_{i-1,{\bar{k}}},\eta} \subseteq \mathcal{O}_{X_{i,{\bar{k}}},\eta}$ for every $i \leq \gamma(X/k)$, we have
$$
\gamma(X/k) \leq {\rm length}_{\mathcal{O}_{X_{\bar k},\eta}}(\mathcal{O}_{Y,\eta}/\mathcal{O}_{X_{\bar k},\eta})
= {\rm length}_{\mathcal{O}_{D,\eta}}(\mathcal{O}_{C,\eta}/\mathcal{O}_{D,\eta}) = \delta(X_{\bar{k}}/\bar{k}),$$
as claimed. 
\end{proof}

\begin{proposition} \label{c-lleq1}
Let $X$ be a geometrically integral normal Gorenstein variety. Then, $\ell_F(X/k) \leq \gamma(X/k) \leq \delta(X_{\bar{k}}/{\bar k}) = \max_{\eta \in D} {\rm length}_{\mathcal{O}_{D,\eta}}(\mathcal{O}_{D, \eta})$.
In particular, if every component of $D$ is reduced,  then $\ell_F(X/k) \leq 1.$
\end{proposition}

\begin{proof}
By \autoref{p-delta}, we only have to show the last equality.
Let $\eta$ be the generic point of an irreducible component of the conductor $D \subset X$. The Gorenstein condition implies $\length_{\mathcal{O}_{D,\eta}}{\mathcal{O}_{C, \eta}}=2\length_{\mathcal{O}_{D,\eta}}{\mathcal{O}_{D, \eta}}$ by  \cite[Proposition A.2]{FS20}, which shows that $\delta(X_{\overline{k}}/\overline{k})= \max_{\eta \in D} \length_{\mathcal{O}_{D,\eta}}{\mathcal{O}_{D, \eta}}$ by \cite[\href{https://stacks.math.columbia.edu/tag/00IV}{Tag 00IV}]{stacks-project}, as claimed.

The last statement is immediate as $\length_{\mathcal{O}_{D,\eta}}{\mathcal{O}_{D, \eta}}=1$ if $D$ is reduced.
% Let $\eta$ be the generic point of a component of the conductor $D \subset X$, thus we conclude by \autoref{p-delta}. \GM{I don't understand thi sentence.}\FB{me neither}
% By \cite[Proposition A.2]{FS20} the Gorenstein condition implies $\length_{\mathcal{O}_{D,\eta}}{\mathcal{O}_{C, \eta}}=2\length_{\mathcal{O}_{D,\eta}}{\mathcal{O}_{D, \eta}}$, which shows that $\delta(X_{\overline{k}}/\overline{k})=\length_{\mathcal{O}_{D,\eta}}{\mathcal{O}_{D, \eta}}$ by \cite[\href{https://stacks.math.columbia.edu/tag/00IV}{Tag 00IV}]{stacks-project}. 
% The last statement is immediate as $\length_{\mathcal{O}_{D,\eta}}{\mathcal{O}_{D, \eta}}=1$ if $D$ is reduced.
\end{proof}

We can improve the bounds of \cite{BT22} in the geometrically integral case.

\begin{corollary}\label{c-length-frob}
Let $X$ be a geometrically integral normal Gorenstein del Pezzo surface. 
Then $\ell_F(X/k) \leq 1.$
Moreover, if $L$ is a torsion line bundle, then $L^{\otimes p} \cong \mathcal{O}_X$. 
In particular, $\Pic^0_{X_{\bar{k}}/\bar{k}} \cong \mathbb{G}_{a,\bar{k}}^{h^1(X,\mathcal{O}_X)}$.
\end{corollary}

\begin{proof}
By \autoref{t-recap-reid}, the conductor $D$ is reduced and thus
we can apply \autoref{c-lleq1} to conclude.
The proof of the assertion on torsion line bundles follows as in \cite[Theorem 4.1]{BT22}.
For the last statement, by \autoref{prop: picard-dP} the Picard scheme $\Pic^0_{X_{\bar{k}}/\bar{k}}$ is a smooth commutative unipotent algebraic group of dimension $h^1(X,\mathcal{O}_X)$.
As it is annihilated by $p$, we conclude by \cite[Proposition VII.11]{Ser88}.
\end{proof}

The previous analysis allows to obtain better estimates for the global generation than \autoref{t-very-ample} in the geometrically integral canonical case.

\begin{corollary}\label{c-gg-dp-int}
Let $X$ be a geometrically integral canonical del Pezzo surface. Let $A$ be a big and nef Cartier divisor on $X$ and suppose $X$ is not geometrically normal. Then $p \in \{2,3\}$ and the following hold:
\begin{enumerate}
    \item If $p=3$, then $\mathcal{O}_X(3A)$ is globally generated and $\omega_X^{- 9}$ is very ample;
    \item If $p=2$, then $\mathcal{O}_X(2A)$ is globally generated and $\omega_X^{- 7}$ is very ample.
\end{enumerate}
\end{corollary}

\begin{proof}
By \autoref{c-length-frob}, $\ell_F(X/k) = 1$ and we conclude by combining \autoref{l-gg-dp} and  \autoref{t-very-ample}.
\end{proof}

\subsection{Anti-pluricanonical maps of non-normal del Pezzo surfaces}

In this section we assume $k$ is algebraically closed. 
Let $X$ be a non-normal integral Gorenstein del Pezzo surface with normalization $\nu \colon Y \to X$ with ramification $C \subseteq Y$ and conductor $D \subseteq X$.
As Gorenstein del Pezzo surfaces have the property $(S_2)$, by \cite[Theorem, Section 2.6]{Rei94}, there is an exact sequence
$$
0 \to \omega_X \to \nu_* \nu^* \omega_X \overset{\Tr \circ \Res}{\to} \omega_D \to 0,
$$
where $\Res$ is the pushforward of the classical residue map $\omega_Y(C) \to \omega_C$ (where we identify $\omega_Y(C) \cong \nu^* \omega_X$ and $\omega_Y(C)|_{C} \cong \omega_C$ by adjunction). 
The homomorphism $\Tr$ is the trace map which, over the generic point $\eta$ of $D$, is given by the $(\mathcal{O}_{D,\eta})$-dual of the inclusion $\mathcal{O}_{D,\eta} \subseteq \nu_* \mathcal{O}_{C,\eta}$ by \cite[Remark 2.9]{Rei94}.
Tensoring with $\omega_X^{-(n+1)}$ and applying the projection formula we obtain
$$
0 \to \omega_X^{-n} \to \nu_* \nu^* \omega_X^{- n} \to \omega_D \otimes \omega_X^{- (n+1)} \to 0.
$$
As $\nu_* \nu^* \omega_X^{-\otimes n}$ is canonically isomorphic to $\nu_*(\omega_Y^{-\otimes n}(-nC))$, taking global sections, we deduce the following:
\begin{lemma} \label{p-global-sections-KX}
We have the following equality of subspaces of $H^0(Y,\omega_Y^{- n}(-nC))$:
$$
\nu^* H^0(X,\omega_X^{- n}) = \Ker(H^0((\Tr \circ \Res) \otimes \omega_X^{\otimes -{ (n+1)}}))
$$
\end{lemma}

We now prove a useful lower bound on the dimension of the space of anti-pluricanonical sections on del Pezzo surfaces.
It will be the main tool to bound the irregularity of del Pezzo surfaces.

\begin{corollary} \label{c-linsystem}
There is an inclusion of $k$-vector spaces:
$$
V := \{s \in H^0(Y,\mathcal{O}_Y(-n(K_Y + C))) \mid s|_{C} = 0 \} \subseteq \nu^* H^0(X,\omega_X^{-n}).
$$
Thus if $\omega_X^{- n}$ is globally generated, then
$$
h^0(X,\omega_X^{- n}) \geq  \dim V + 2.
$$
\end{corollary}
\begin{proof}
By the natural identifiation $\omega_Y(C)|_C \cong \omega_C$ given by adjunction, the space $V$ is equal to the kernel of the homomorphism $H^0(\Res \otimes \omega_X^{-n-1})$, hence it is contained in the kernel of $H^0((\Tr \circ \Res) \otimes \omega_X^{-n-1}) = \nu^* H^0(X,\omega_X^{-n})$ by \autoref{p-global-sections-KX}. 

If $\omega_{X}^{-n}$ is globally generated, then the linear system $|\nu^* H^0(X,\omega_X^{-n})|$ has no base points on $C$. 
Since all sections in $V$ vanish on $C$ and $\omega_X^{-1}$ is ample there are at least two more linearly independent sections of $\nu^* H^0(X,\omega_X^{-n})$ that are non-zero when restricted to $C$, thus concluding the inequality.
\end{proof}

\subsection{Irregularity of geometrically integral l.c.i. del Pezzo surfaces}

We prove effective bounds on the values of the irregularity of locally complete intersection (lci) del Pezzo surfaces.

\begin{proposition} \label{p-bound-irregularity-normal}
Let $X$ be a geometrically integral normal locally complete intersection del Pezzo surface over a field $k$ of characteristic $p>0$. 
Let $\nu \colon Y \to X_{\bar{k}}$ be the normalization of $X_{\bar{k}}$ and let $C \subseteq Y$ be the ramification of $\nu$. Then, one of the following holds:
\begin{enumerate}
    \item $h^1(X,\omega_X^{ n}) = 0$ for all $n \in \mathbb{Z}$.
    \item $p = 3$, $(Y,C) = (\mathbb{P}^2,2L)$, $h^1(X,\mathcal{O}_X) = 2$, and $h^1(X,\omega_X^{ n}) = 0$ for all $n \geq 2$.
    % \item $p = 2$, $(Y,C) = (\mathbb{P}^2,2L)$, $h^1(X,\mathcal{O}_X) = 3$, and $h^1(X,\omega_X^{\otimes n}) = 0$ for all $n \geq 4$.
    \item $p = 2$, $(Y,C) = (\mathbb{P}^2,2L)$, $h^1(X,\mathcal{O}_X) = 1$, and $h^1(X,\omega_X^{ n}) = 0$ for all $n \geq 2$.
    \item $p = 2$, $(Y,C) = (\mathbb{P}(1,1,2),2L)$, $h^1(X,\mathcal{O}_X) = 1$, and $h^1(X,\omega_X^{ n}) = 0$ for all $n \geq 2$.
\end{enumerate}
\end{proposition}

\begin{proof}
If $X$ is geometrically normal, then $X$ is geometrically canonical by \autoref{prop: geom-normal-canonical-dP}. By Serre duality, it is sufficient to show that $h^1(X_{\overline{k}}, \omega_{X_{\overline{k}}}^{\otimes n})=0$ for $n > 0$. This follows from \cite[Theorem 5.6.a]{Berkvvfails}.
%(1) holds by \cite[Proposition 4.2]{HW81}.
If $X$ is not geometrically normal and the ramification divisor contains a reduced component, then $X_{\overline{k}}$ is tame and $h^1(X,\omega_X^{\otimes n}) = 0$ for all $n \in \mathbb{Z}$ by \cite[Corollary 4.10]{Rei94}. 
Therefore, by \cite[Theorem 4.1]{PW}, we may assume that $p \in \{2,3\}$, $h^1(X,\mathcal{O}_X) > 0$ and $(Y,C) = (\mathbb{P}(1,1,d),2L)$ for some $d \geq 1$ where $L$ is a line through the vertex of the cone.

Choose weighted coordinates $x,y,z$ of degree $1,1,d$ on $Y$ such that $L=\left\{x = 0 \right\}$, hence $2L=\left\{x^2=0 \right\}$ in weighted coordinates. 
Let $n \geq 1$ and $V_{n,d} \subseteq H^0(\mathbb{P}(1,1,d),\mathcal{O}(nd))$ be the subspace of sections vanishing along $2L$. 
Then, $V_{n,d}$ consists of weighted homogeneous polynomials of the form $x^2f_{nd-2}(x,y,z)$, hence 
$\dim V_{n,d} = \sum_{j=1}^{n} (jd - 1) = \frac{n^2 + n}{2}d - n.$
As $\nu^* \omega_{X_{\bar{k}}} \cong \mathcal{O}(-dL)$, we have $\nu^* \omega_{X_{\bar{k}}}^{-n} \cong \mathcal{O}(ndL)$.  
By \autoref{c-linsystem} we have 
\begin{equation} \label{e-VDNbound1}
h^0(X,\omega_X^{- n}) \geq 
\begin{cases}
\frac{n^2 + n}{2}d - n. \\
\frac{n^2 + n}{2}d - n+2 \text{ if, additionally, $\omega_X^{-n}$ is globally generated }
\end{cases}
\end{equation}
By the Riemann--Roch formula \cite[Theorem 2.10]{Tan18b} we have
$$ h^0(X,\omega_X^{- n})-h^1(X, \omega_X^{-n})=1-h^1(X,\mathcal{O}_X)+\frac{n^2 + n}{2}K_X^2=1-h^1(X,\mathcal{O}_X)+ \frac{n^2 + n}{2}d.$$
Thus if we assume $h^1(X, \omega_X^{- n})=0$ we deduce from Equation \eqref{e-VDNbound1} that
\begin{equation} \label{e-VDNbound2}
h^1(X,\mathcal{O}_X) = 1 - h^0(X,\omega_X^{- n}) + \frac{n^2 + n}{2}d \leq
\begin{cases}
n + 1 \\
n - 1 \text{ if, additionally, $\omega_X^{- n}$ is g.g. }
\end{cases}
\end{equation}

We also recall Maddock's bound \cite[Corollary 1.2.6]{Mad16}:
if $h^1(X,\omega_X^{ n}) \neq 0$ but $h^1(X,\omega_X^{pn}) = 0$, then
\begin{equation} \label{e-madbound}
h^1(X,\mathcal{O}_X) \geq \frac{nd(p-1)(3 + n(2p-1))}{12}.
\end{equation}

Now assume $p = 3$. By Serre vanishing and $h^1(\mathcal{O}_X) \neq 0$, there exists a largest $N \geq 0$ such that $h^1(X,\omega_X^{- N})  = h^1(X,\omega_X^{ (N + 1)}) \neq 0$.  
By \eqref{e-VDNbound2} and \eqref{e-madbound}, we have the following chain
$$
N + 2 \geq h^1(X,\mathcal{O}_X) \geq \frac{(N+1)d(p-1)(3 + (N+1)(2p-1))}{12} = \frac{(N+1)d(8 + 5N)}{6},
$$
hence $N = 0$, $d = 1$, showing $h^1(X, \mathcal{O}_X) \leq 2$. Finally $h^1(X,\mathcal{O}_X) = 2$ by \autoref{t-recap-reid}.

Now assume $p = 2$. Then, the argument of the previous paragraph yields
\begin{equation}\label{e-bigN}
N + 2 \geq h^1(X,\mathcal{O}_X) \geq \frac{(N+1)d(p-1)(3 + (N+1)(2p-1))}{12} = \frac{(N+1)d(N+2)}{4},
\end{equation}
hence $N \leq 3$.
% $(N,d) \in \{(3,1),(2,1),(1,2),(1,1),(0,4),(0,3),(0,2),(0,1)\}$. 
Therefore, $h^1(X,\omega_X^{- 4}) = 0$ and, by \autoref{c-gg-dp-int}, $\omega_X^{-  4}$ is globally generated, so $h^1(X,\mathcal{O}_X) \leq 3$ by \eqref{e-VDNbound2}.
If $h^1(X,\mathcal{O}_X) = 1$, then $N = 0$ and $d \in \{1,2\}$ by \eqref{e-bigN} and we get Cases (3) and (4).

So, it remains to exclude the possibility $h^1(X,\mathcal{O}_X) \geq 2$ in characteristic $p = 2$.
By \autoref{c-gg-dp-int}, $\omega_X^{- 2}$ is globally generated, so by \eqref{e-bigN} the inequality $h^1(X,\mathcal{O}_X) \geq 2$
implies $N = 2$, $d = 1$, and $h^1(X,\mathcal{O}_X) = 3$.

Seeking a contradiction, assume that there exists an $X$ with these invariants. Since $N = 2$, we have $H^1(X, \omega_X^{3}) \neq 0$ and $H^1(X, \omega_X^6)=0$. Let $Z \to X$ be a non-trivial $\alpha_{\omega_X^3}$-torsor. and let $k_Z:=H^0(Z, \mathcal{O}_Z)$.
    Note that $Z$ is an l.c.i. del Pezzo surface by \cite[Theorem 1.2.3]{Mad16}
    and by \cite[Equation (1.2.4)]{Mad16}, we have 
    $[k_Z :k] (1-h^1(Z, \mathcal{O}_Z)) = 2$ so that we have $[k_Z :k]=2$ and $H^1(Z,\mathcal{O}_Z) = 0$.
    By \cite[Equation (1.2.5)]{Mad16} we then conclude that $K_Z^2=16$. 
    Now, we consider the following diagram:
$$
\xymatrix{
(Z_{\bar{k}})^{\norm} \ar[r] \ar^{(\pi_{\bar{k}})^{\norm}}[d] & Z_{\bar{k}} \ar[r] \ar^{\pi_{\bar{k}}}[d] & Z \ar^{\pi}[d] \ar^{f}[dr] & \\
(X_{\bar{k}})^{\norm} \ar[r]  & X_{\bar{k}} \ar[r] \ar[d] & X_{k_Z} \ar[r] \ar[d] & X \ar[d] \\
& \Spec \bar{k} \ar[r] & \Spec k_Z \ar[r] & \Spec k,
}
$$
where $Z_{\bar{k}} = Z \times_{\Spec k_Z} \Spec \bar{k}$.
Since $f$ and $X_{k_Z} \to X$ are finite of degree $2$, the morphism $\pi$ is finite and birational. In particular, $Z$, considered as a $k_Z$-scheme, is geometrically integral and the induced map $(\pi_{\bar{k}})^{\norm}$ of the normalisations is an isomorphism.
As $(X_{\bar{k}})^{\norm} \cong \mathbb{P}^2$, $16=K_Z^2 \leq K_{(Z_{\bar{k}})^{\norm}}^2=9$, reaching a contradiction.
% If $h^1(X,\mathcal{O}_X) \geq 2$, then, as $\omega_X^{- 2}$ is globally generated by \autoref{c-gg-dp-int}, we have $h^1(X,\omega_X^{- 2}) \neq 0$, hence $N \geq 2$. By \eqref{e-bigN}, this implies $N = 2$, $d = 1$, and $h^1(X,\mathcal{O}_X) = 3$.
%  If $h^1(X,\mathcal{O}_X) = 1$, then $N = 0$ and $d \in \{1,2\}$ by \eqref{e-bigN}.
\end{proof}

\begin{corollary}
Let $X$ be a geometrically integral normal locally complete intersection del Pezzo surface over a field $k$ of characteristic $p$. Let $\nu \colon Y \to X_{\bar{k}}$ be the normalization of $X_{\bar{k}}$ and let $C \subseteq Y$ be the ramification of $\nu$. 
Then, one of the following holds:
\begin{enumerate}
    \item $h^1(X,\mathcal{O}_X) = 0$,  $K_X^2 \geq 3$, and $\omega_X^{-1}$ is very ample.
    \item $h^1(X,\mathcal{O}_X) = 0$, $K_X^2 = 2$, and $\omega_X^{-2}$ is very ample.
    \item $h^1(X,\mathcal{O}_X) = 0$, $K_X^2 = 1$, and $\omega_X^{-3}$ is very ample.
    \item $h^1(X,\mathcal{O}_X) = 1$, $p=2$, $(Y,C, K_X^2) \in \{(\mathbb{P}^2,2L, 1),(\mathbb{P}(1,1,2),2L,2)\}$,  and $\omega_X^{-6}$ is very ample.
    \item $h^1(X,\mathcal{O}_X) = 2$, $p=3$, $(Y,C,K_X^2) = (\mathbb{P}^2,2L,1)$, and $\omega_X^{-7}$ is very ample.
\end{enumerate}

\end{corollary}
\begin{proof}
Claims (1), (2), and (3) are a consequence of \cite[Proposition 2.14]{BT22}
%\cite[Corollary 4.5]{HW81} 
if $X$ is geometrically normal, hence geometrically canonical, and \cite[Corollary 4.10]{Rei94} if $X$ is not geometrically normal.

Let us prove Claim (4).  By \autoref{p-bound-irregularity-normal}, we have $p=2$ and the desired classification of $(Y,C)$. By \autoref{l-gg-dp}, $\omega_X^{-2}$ is globally generated. Using \autoref{p-bound-irregularity-normal}, it is easy to check that $\omega_X^{-4}$ is $0$-regular with respect to $\omega_X^{-2}$.
Claim (5) is proven similarly.
\end{proof}

\subsection{Refinements in the regular and canonical case}

We show various refinements of the bounds of \autoref{p-bound-irregularity-normal} in the case where we assume $X$ to be a  regular or canonical del Pezzo surface.
We start with the case $p=3$.

\begin{proposition} \label{p-wildp=3}
Let $X$ be a geometrically integral canonical del Pezzo surface over a field $k$ of characteristic $p = 3$. Then, $X$ is tame.
\end{proposition}
\begin{proof}
Without loss of generality, we may assume that $k$ is separably closed. 

First, assume that $X$ is regular.
Seeking a contradiction, we assume that $h^1(X, \mathcal{O}_X) \neq 0$. 
Let $\nu \colon Y \to X_{\bar{k}}$ be the normalisation of $X_{\bar{k}}$ and let $C \subseteq Y$ be the ramification of $\nu$.
By \autoref{p-bound-irregularity-normal} and Serre duality, we know that $h^1(X,\mathcal{O}_X) = 2$, $K_X^2 = 1$, $h^1(X,\omega_X^{-n}) = 0$ for $n > 0$, and $(Y,C) = (\mathbb{P}^2,2L)$. 

First, we claim that $h^0(X,\omega_X^{-n} \otimes \mathcal{L}) = 0$ for all non-trivial torsion line bundles $\mathcal{L}$ and for $n \in \left\{0, 1\right\}$. 
Since $X$ is reduced, this holds if $n = 0$.
For the case $n=1$, by the Riemann--Roch theorem we have
$$
\chi(\omega_X^{-1} \otimes \mathcal{L}) = 0,
$$
so if $h^1(X,\omega_X^{-1} \otimes \mathcal{L}) \neq 0$, then $h^1(X,\omega_X^{2} \otimes \mathcal{L}^{-1}) \neq 0$ by Serre duality. 
Since $\omega_X^{6} \otimes \mathcal{L}^{-3} \cong \omega_X^{6}$ by \autoref{c-length-frob} and $h^1(X,\omega_X^{6}) = 0$ by \autoref{p-bound-irregularity-normal}, there exists a non-trivial $\alpha_{(\omega_X^{2} \otimes \mathcal{L}^{-1})}$-torsor $Z \to X$ such that $Z$ is an l.c.i. del Pezzo surface. Moreover, by \cite[Equation (1.2.5)]{Mad16}, we have
$
2^e(1-q_Z)= 10
$
for some integers $0 \leq e \leq 1$ and $q_Z \geq 0$, contradicting our assumption.

By Riemann--Roch, we have $h^0(X,\omega_X^{-2}) = 2$.
Write
$$
|-2K_X| = F + |M|
$$
where $F$ is the fixed part and $M$ is the movable part of the linear system. Since $M \neq 0$, we have $F \in |-nK_X + E|$ for some $0 \leq n \leq 1$ and a divisor $E$ such that $\mathcal{O}_X(E)$ is torsion. By the previous paragraph, we have $h^0(X,\omega_X^{-n}(E)) = 0$, hence $F = 0$.

 Since the linear system $|-2K_X|$ does not have fixed components, its base locus $Z$ is $0$-dimensional and we denote by $A$ the ring of global section $H^0(Z, \mathcal{O}_Z)$. 
 Since $(-2K_X)^2 = 4$, we have ${\rm length}_k(A) = 4$, so $A$ is an Artinian $k$-algebra of length $4$. As $k$ is separably closed, we can write $A = \prod_{i=1}^s A_i$ where each $A_i$ is a local Artinian $k$-algebra of dimension $n_i$ over its residue field $k_i$ and $k_i$ is a purely inseparable extension of $k$.
Since 
$$
4 = {\rm length}_k(A) = \sum_{i=1}^s {\rm length}_k(A_i) = \sum_{i=1}^s n_i [k_i:k]
$$
and $p = 3$, we have $k_i = k$ for at least one $i$. In other words, at least one of the base points of $|-2K_X|$ is a $k$-rational point $P$.

If $P$ is in the image of the conductor $D$ under the natural map $X_{\bar{k}} \to X$, then $P$ lies in the non-smooth locus of $X \to \Spec k$, so $X$ cannot be regular at $P$ by \cite[Corollary 2.6]{FS20}. Hence, in this case, the proof is finished.

So, seeking a contradiction, assume that $P$ is not in the image of $D$. Let $P'$ be the unique preimage of $P$ under the map $Y \to X_{\bar{k}} \to X$. Since $\nu$ is an isomorphism around $P'$, the point $P'$ lies in the base locus of $\nu^* |-2K_{X_{\bar{k}}}|$. By \autoref{c-linsystem}, we know that $C = 2L \in \nu^* |-2K_{X_{\bar{k}}}|$, hence $P' \in C$ and thus $P$ is in the image of $D$ under $X_{\bar{k}} \to X$. This contradicts our choice of $P$.

Finally, assume that $X$ is canonical.
By \autoref{p-klt-rational} we can replace $X$ with its minimal resolution, which is a regular weak del Pezzo surface. 
By running a $K_X$-MMP we can suppose $X$ is a weak del Pezzo surface admitting a Mori fibre space structure $\pi \colon X \to B$. 
If $X$ is a regular del Pezzo surface, we conclude by the previous case.  
If $B$ is a curve and $X$ is a weak del Pezzo surface, then the generic fibre $F$ is a regular conic. 
As $p = 3$, $F$ is smooth by \cite[Lemma 2.17]{BT22} and thus $-K_B$ is ample by \cite[Corollary 4.10.c]{Eji19}.
Therefore $H^1(B, \mathcal{O}_B)=0$ and,
as $H^1(X, \mathcal{O}_X)=H^1(B, \mathcal{O}_B)=0$ by the relative Kawamata--Viehweg vanishing theorem \cite[Theorem 4.2]{Tan18b}, we conclude.
\end{proof}

In the following proposition, we describe the geometry of $\alpha_{\omega_X}$-torsors over wild regular del Pezzo surfaces in characteristic $p = 2$. The reader should compare this with the construction of the regular wild del Pezzo surfaces of degree $1$ in \cite{Mad16}.

\begin{proposition} \label{p-wildp=2}
Let $X$ be a geometrically integral regular del Pezzo surface over a field $k$ of characteristic $p = 2$. Assume that $h^1(X,\mathcal{O}_X) \neq 0$. 
Then, $\pdeg(k) \geq 2$, $K_X^2 \leq 2$, and there exists an $\alpha_{\omega_X}$-torsor $Z \to X$ such that $Z$ satisfies the following properties:

\begin{enumerate}
\item If $K_X^2 = 2$, then $k_Z \coloneqq h^0(Z,\mathcal{O}_Z)$ is a purely inseparable extension of $k$ of degree $2$ and $Z$ is a twisted form of $\mathbb{P}(1,1,2)$ over $k_Z$.
\item If $K_X^2=1$, then $Z$ is a normal tame del Pezzo surface such that $\epsilon(Z/k)=1$, $K_Z^2=8$ and the normalised base change $((Z_{\overline{k}})_{\red}^{\norm}, E)$ is $(\mathbb{P}^2,L)$.
\end{enumerate}

\end{proposition}

\begin{proof}
If $X$ is not tame, then $\rho(X)=1$ by \autoref{p-bound-irregularity-normal}. If $\pdeg(k)=1$, then $X$ is geometrically canonical by \cite{FS20}, contradicting $h^1(X, \mathcal{O}_X) \neq 0$.

By \autoref{p-bound-irregularity-normal}, we have $h^1(X,\mathcal{O}_X) = h^1(X,\omega_X) = 1$, $h^1(X,\omega_X^{n}) = 0$ for $n \geq 2$, and $K_X^2 \in \{1,2\}$. In particular, there exists a non-trivial $\alpha_{\omega_X}$-torsor $f \colon Z \to X$. In the following, we treat the cases $K_X^2 = 2$ and $K_X^2 = 1$ separately. We set $k_Z \coloneqq H^0(Z,\mathcal{O}_Z)$. Note that, by the same proof as for the second paragraph of the proof of \autoref{p-wildp=3}, we have $h^1(X,\omega_X^{n} \otimes \mathcal{L}) = 0$ for all torsion line bundles $\mathcal{L}$ and all $n \geq 2$. 

Assume $K_X^2 = 2$. In this case, by \cite[Equation (1.2.5)]{Mad16}, we have $[k_Z:k] = 2$, $h^1(Z,\mathcal{O}_Z) = 0$, and $K_Z^2 = 8$, where we compute the self-intersection over $k_Z$. Now, consider the commutative diagram
$$
\xymatrix{
(Z_{\bar{k}})^{\norm} \ar[r] \ar^{(\pi_{\bar{k}})^{\norm}}[d] & Z_{\bar{k}} \ar[r] \ar^{\pi_{\bar{k}}}[d] & Z \ar^{\pi}[d] \ar^{f}[dr] & \\
(X_{\bar{k}})^{\norm} \ar[r]  & X_{\bar{k}} \ar[r] \ar[d] & X_{k_Z} \ar[r] \ar[d] & X \ar[d] \\
& \Spec \bar{k} \ar[r] & \Spec k_Z \ar[r] & \Spec k,
}
$$
where $Z_{\bar{k}} = Z \times_{\Spec k_Z} \Spec \bar{k}$. As in the end of the proof of \autoref{p-bound-irregularity-normal}, the diagram shows that $Z$ is geometrically integral when considered as a $k_Z$-scheme and $(\pi_{\bar{k}})^{\norm}$ is an isomorphism, hence $(Z_{\bar{k}})^{\norm} \cong (X_{\bar{k}})^{\norm} \cong \mathbb{P}(1, 1, 2)$ by \autoref{p-bound-irregularity-normal}. 
In particular, we have $K_Z^2 = 8 = K^2_{(Z_{\bar{k}})^{\norm}}$, so $Z$ is in fact geometrically normal. Therefore, $Z_{\bar{k}} \cong \mathbb{P}(1,1,2)$, so $Z$ is a twisted form of $\mathbb{P}(1,1,2)$ over $k_Z$.

Assume $K_X^2=1$. In this case, by \cite[Equation (1.2.5)]{Mad16} we have $k_Z = k$ and $h^1(Z, \mathcal{O}_Z)=0$. 
We first claim that $Z$ is not geometrically reduced.
Indeed, let $\psi \colon \mathbb{P}^2 \to X$ be the normalised base change and consider the $\alpha_{\psi^*\omega_X}$-torsor $T:=Z \times_{X} \mathbb{P}^2 \to \mathbb{P}^2$ obtained by base changing along $\psi$. As $\psi^*\omega_X$ is anti-ample, considering the exact sequence $$0 = H^0(\mathbb{P}^2, \psi^*\omega_X^p) \to H^1_{{\fppf}}(\mathbb{P}^2, \alpha_{\psi^*\omega_X}) \to H^1(\mathbb{P}^2, \psi^*\omega_X) = 0$$ we have that $H^1_{{\fppf}}(\mathbb{P}^2, \alpha_{\psi^*\omega_X})=0$, so that $T \to \mathbb{P}^2$ is a trivial torsor and thus $T$ is not reduced.
Since $T \to Z_{\bar{k}}$ is generically an isomorphism, $Z$ is not geometrically reduced.

Next, we claim that $Z$ is normal. 
Suppose by contradiction it is not and let $\nu \colon Z^{\norm} \to Z$ be the normalisation. In this case, as $\nu$ is not an isomorphism, we have that $K_{Z^{\norm}}^2>8$, where we calculate the self-intersection number over $k$. Let $g \colon Z^{\norm} \to X$ be the composition $f \circ \nu$. As $X$ is regular and $Z^{\norm}$ is integral, by \cite[Proposition 2.2.1]{Mad16} there exists a line bundle $\mathcal{M}$ numerically equivalent to $mK_X$ for some integer $m$ such that $g$ is a non-trivial $\alpha_\mathcal{M}$-torsor. In particular, $Z^{\norm}$ is Gorenstein and thus, as it is the normalization of a del Pezzo surface, $Z^{\norm}$ is also a del Pezzo surface. Therefore, $m \geq 0$.
We now distinguish two cases:
\begin{enumerate}
    \item if $H^0(Z^{\norm}, \mathcal{O}_{Z^{\norm}}) =k$, by \cite[Equation (1.2.4)]{Mad16} we have $2(1+m)^2 = K_{Z^{\norm}}^2>8$, which implies that $m>1$,  contradicting \autoref{p-bound-irregularity-normal};
    \item if $[H^0(Z^{\norm}, \mathcal{O}_{Z^{\norm}}):k ]=2$, we have that $(1+m)^2=K_{Z^{\norm}}^2>4$, which also implies that $m>1$,  contradicting \autoref{p-bound-irregularity-normal} as well. Note that, here, we calculate the self-intersection $K_{Z^{\norm}}^2$ over $H^0(Z^{\norm}, \mathcal{O}_{Z^{\norm}})$. 
\end{enumerate}
Thus, $Z$ must be normal.

By \cite[Example 1]{Kle66}, we have that $8=K_Z^2=2^{\epsilon(Z/k)}(K_{(Z_{\bar{k}})_{\red}^{\norm}} + E)^2 $.
Since $Z$ is geometrically non-reduced, we have $\epsilon(Z/k) \geq 1$, and since $f$ is finite flat of degree $2$ and $\epsilon(X/k) = 0$, we have $\epsilon(Z/k) \leq 1$, hence $\epsilon(Z/k) = 1$.
As $Z$ is normal and Gorenstein, we can apply \cite[Theorem 4.1]{PW} to conclude that $Z=\mathbb{P}^2$ and $E$ is a line.
\end{proof}

\begin{proof}[Proof of \autoref{t-intro-irr}]
Combine \autoref{p-bound-irregularity-normal}, \autoref{p-wildp=3} and \autoref{p-wildp=2}.
\end{proof}

%\subsection{Refinements in $p$-degree one}

%In the case where the $p$-degree is at most one, we can further improve our results using the main result of \cite{FS20}.

%\begin{corollary}\label{p-sharp-tameness-3}
%Let $k$ be a field of characteristic $p>0$ with $\pdeg(k) \leq 1$.
%\begin{enumerate}
 %\item If $p \geq 3$, then weak canonical del Pezzo surfaces $X$ over $k$ are tame;
 %\item if $p=2$, then regular del Pezzo surfaces are tame.
%\end{enumerate}
%\end{corollary}

%\begin{proof}
%We first discuss the case $p \geq 3$.
%By \autoref{p-klt-rational} we can replace $X$ with its minimal resolution. 
%By running a $K_X$-MMP we can suppose it admits a Mori fibre space structure $\pi \colon X \to B$. 
%If $B$ is a point, then $\rho(X)=1$ and therefore $Y$ is geometrically normal by \cite[Theorem 14.1]{FS20} and thus tame.  
%If $B$ is a curve, then the generic fibre $F$ is a regular conic.  As $p \geq 3$, $F$ is smooth by \cite[Lemma 2.17]{BT22} and thus $-K_B$ is ample by \cite[Corollary 4.10.c]{Eji19}. Therefore $H^1(B, \mathcal{O}_B)=0$ and, as $H^1(X, \mathcal{O}_X)=H^1(B, \mathcal{O}_B)=0$ by relative Kawamata--Viehweg vanishing theorem \cite[Theorem 4.2]{Tan18b}, we conclude.

%We now discuss the case $p=2$. If $X$ is not tame, then $\rho(X)=1$ by \autoref{p-bound-irregularity-normal}. 
%Therefore $X$ is geometrically normal by \cite{FS20}, reaching a contradiction.
%\end{proof}

\section{On the BAB conjecture for surfaces over arbitrary fields}
\label{s-BAB}

In this section, we prove boundedness results for del Pezzo surfaces over arbitrary fields.

We recall some terminology when discussing boundedness in birational geometry.
For the following definition,
we say that a scheme $X$ is a \emph{projective variety} if $X$ is integral, $H^0(X, \mathcal{O}_X)$ is a field $k$ and the natural morphism $\pi_X \colon X \to \Spec(k)$ is projective. 
We will always consider $X$ as a $k$-variety via the natural morphism $\pi_X$.

\begin{definition}
We say that a class of projective varieties $\mathcal{X}$ is \emph{bounded} (resp. \emph{birationally bounded}) if there exists a projective flat morphism $Y \to T$ of finite type $\mathbb{Z}$-schemes such that for every $X \in \mathcal{X}$ with $k \coloneqq H^0(X,\mathcal{O}_X)$, there exists a morphism $\Spec(k) \to T$ and a $k$-isomorphism (resp. a $k$-birational map) $X \to Y \times_V \Spec(k)$. 
\end{definition}

The Borisov--Alexeev--Borisov (BAB) conjecture states that mildly singular Fano varieties form a bounded family in every dimension.

\begin{conjecture}[BAB]\label{BAB-conj}
For any rational number $\varepsilon>0$, the class
$$\mathcal{X}_{{d, \varepsilon}} = \left\{X \mid X \emph{ is a } 
\varepsilon \emph{-klt Fano variety of dimension } d \right\} $$
is bounded.
\end{conjecture}

\begin{remark}
The presence of the $\varepsilon$--klt hypothesis is necessary in the BAB conjecture, already in dimension 2.
Indeed, Gorenstein del Pezzo surfaces with general log canonical singularities are not bounded as cones over elliptic curves of \autoref{l-HW} show.
Moreover, boundedness already fails for klt del Pezzo surfaces as the set of weighted projective planes $\left\{\mathbb{P}(1,1,d)\right\}_{d \geq 1}$ shows.
\end{remark}

We discuss the BAB conjecture for surfaces defined over arbitrary fields.
The result of \cite{Ale94, CTW17} shows that the class of geometrically $\varepsilon$-klt del Pezzo surfaces form a bounded family. 
However, the conjecture is still open for $\varepsilon$-klt del Pezzo surfaces defined 
over an imperfect field.

In \autoref{ss-canonical}, we settle the BAB conjecture for geometrically integral canonical del Pezzo surfaces.
In the remaining two subsections, we discuss the general $\varepsilon$-klt case. 
In \autoref{ss-bound-vol-eps} prove boundedness of the anticanonical volumes for $\varepsilon$-klt del Pezzo surfaces over imperfect fields.
This, together with a bound on the $\mathbb{Q}$-Gorenstein index proven in \autoref{ss-bound-Qindex-eps}, implies the BAB \autoref{BAB-conj} for such surfaces in characteristic $p>5$ (cf. \autoref{thm: BAB_imperfect}).

\subsection{Boundedness of geometrically integral canonical del Pezzo surfaces} \label{ss-canonical}

In \cite{Tan19}, Tanaka proves boundedness for geometrically integral regular del Pezzo surfaces.
As a consequence of our results in \autoref{s-boundedness} and \autoref{s-irregularity}, we can extend Tanaka's result to the canonical case.

\begin{theorem}\label{t-bab-can-dP}
The class 
$$\mathcal{X}_{\textup{dP}, \textup{can}}= \left\{ X \mid X \emph{ is a geometrically integral canonical del Pezzo surface}\right\}$$
is bounded.
\end{theorem}

\begin{proof}
%By \autoref{c-12-v-ample} $\omega_X^{\otimes -12}$ is very ample.
As $X$ is geometrically integral, $\epsilon(X/k)=0$ and thus $K_X^2 \leq 16$ by \autoref{p-bound-volume-canonical-dP}.
Since $X$ is canonical, $K_X$ is Cartier and $K_X^2$ is an integer.
By \autoref{p-bound-irregularity-normal}, \autoref{c-12-v-ample}, and Riemann--Roch, there exists an $N > 0$ such that $\omega_X^{-12}$ embeds $X$ into $\mathbb{P}^n_k$ for some $n \leq N$.
Again by \autoref{p-bound-irregularity-normal}, and since $K_X^2 \leq 16$, the possibilities for the Hilbert polynomial $\chi(X, \omega_X^{ -12})$ are finite.
Therefore, all $X$ arise via pullback from a universal family over a suitable finite union of Hilbert scheme of finite type over $\Spec \mathbb{Z}$ \cite[Theorem 1.4]{Kol96}. 
\end{proof}

\subsection{Bounds on the volume of $\varepsilon$-klt del Pezzo surfaces} \label{ss-bound-vol-eps}

We prove an explicit bound for the volumes of $\varepsilon$-klt del Pezzo surfaces, generalising the results of \cite{Ale94, AM04} to imperfect fields.
To do so, we start with some elementary computations on surfaces of del Pezzo type admitting a Mori fibration onto a curve. Recall the definition of surfaces of del Pezzo type from \autoref{def: dPtype}.

\begin{lemma}\label{lem: bound-self-intersection}
 Let $X$ be a regular surface of del Pezzo type. Let $\pi \colon X \to B$ be a Mori fibre space onto a regular curve and let $F_b=\pi^*b$, where $b$ is a closed point of $B$.
Then, there exists an integral curve $\Gamma$ such that $$\NE(X)=\mathbb{R}_+[F_b]+\mathbb{R}_+[\Gamma]$$

Moreover, setting $d_{\Gamma} \coloneqq [H^0(\Gamma,\mathcal{O}_{\Gamma}):k]$ and $m_{\Gamma} = [k(\Gamma):k(B)]$, there exists $n \geq 0$ such that
\begin{equation}\label{eq: gammasquare}
\Gamma^2=-d_\Gamma \cdot n, \quad K_X \cdot \Gamma=d_\Gamma(n-2),
\end{equation}
    and
    \begin{equation}\label{eq: self-int}
    K_X^2=\frac{d_\Gamma}{m_\Gamma^2} 
    (8m_{\Gamma}+4n(1-m_{\Gamma})) \leq 8.
    \end{equation} 
\end{lemma}

\begin{proof}
The existence of $\Gamma$ is a consequence of the cone theorem \cite[Theorem 2.14]{Tan18b}, while Equation \eqref{eq: gammasquare} is proved in \cite[Lemmas 4.3, 4.6]{BT22}.
To prove Equation \eqref{eq: self-int}, we write $K_X \equiv xF_b+y\Gamma$ for some $x,y \in \mathbb{Q}$. 
Set $d_b = [k(b):k]$.
As $K_X\cdot F_b=-2d_b$ we conclude that $y=-\frac{2}{m_\Gamma}.$  
Therefore
$$d_\Gamma(n-2) = K_X \cdot \Gamma = xm_\Gamma d_b+\frac{2n d_\Gamma}{m_\Gamma},$$
which implies
$$K_X \equiv \frac{d_\Gamma(m_\Gamma(n-2)-2n)}{m_\Gamma^2d_b} F_b - \frac{2}{m_\Gamma}\Gamma.$$
A straightforward computation with intersection numbers then shows Equation \eqref{eq: self-int}.
Finally, as $(1-m_\Gamma) \leq 0$, we have $K_X^2 \leq \frac{d_\Gamma}{m_\Gamma} 8$ and, 
as $d_\Gamma \leq m_\Gamma,$ we conclude.
\end{proof}

We now prove bounds on the anticanonical volume of $\varepsilon$-klt del Pezzo.

\begin{theorem}\label{t-bounds-volume}
Fix a rational number $\varepsilon>0$. 
Then for every geometrically integral $\varepsilon$-klt del Pezzo surface $X$, we have $$K_X^2 \leq \max \left\{ 9, 8+ 20\frac{(1-\varepsilon)^2}{\varepsilon} \right\}.$$
\end{theorem}

\begin{proof}
Let $f\colon Y \to X$ be the minimal resolution and write 
$K_Y+\sum_i b_i E_i = f^*K_X$. 
By the $\varepsilon$-klt hypothesis and minimality of $f$, we have $0 < b_i < 1-\varepsilon.$
We run a $K_Y$-MMP which ends with $\psi \colon Y \to Z$, where $Z$ is a regular projective surface admitting a Mori fibre space structure $\pi \colon Z \to B$.
Since $-(K_Y+\sum_i b_i E_i)$ is big and nef, so is $-(K_Z+\Delta_Z)$, where $\Delta_Z=\psi_* \left( \sum b_i E_i \right)$. 
Moreover, $K_X^2 = (K_Y+\sum_i b_i B_i)^2 \leq (K_Z+\Delta_Z)^2.$

Suppose $\dim(B)=0$. Then $Z$ is a regular del Pezzo surface of Picard rank 1. As $-(K_Z+\Delta_Z)$ is ample, there exists $0 \leq \lambda<1$ such that $\Delta_Z \equiv - \lambda K_Z$. 
Therefore we deduce $(K_Z+\Delta_Z)^2=(1-\lambda)^2K_Z^2 \leq K_Z^2 \leq 9$, where the last inequality follows by \cite[Theorem 1.2]{Tan19}.

Suppose $\dim(B)=1$. Let $\Gamma$ be the extremal curve described in \autoref{lem: bound-self-intersection}.
We write $\Delta_Z=\alpha \Gamma + G$, where $\text{Supp}(G)$ does not contain $\Gamma$. Since the Picard rank of $Z$ is $2$, $G$ is a $\mathbb{Q}$-Cartier nef divisor. As $-(K_Z+\Delta_Z)$ and $-(K_Z+\alpha \Gamma)$ are big and nef classes, their intersection with $G$ is non-positive, and thus we have
$$(K_Z+\Delta_Z)^2 = (K_Z+\alpha \Gamma)^2+ (K_Z+\alpha \Gamma)\cdot G+  (K_Z+\Delta_Z)\cdot G \leq (K_Z+\alpha \Gamma)^2.$$
Therefore, it is sufficient to bound the volume of del Pezzo surface pairs $(Z, \alpha \Gamma)$, where $Z \to B$ is a Mori fibre space onto a curve, $\Gamma$ is the extremal curve of \autoref{lem: bound-self-intersection} and $0 \leq \alpha <1-\varepsilon$.
Note that 
\begin{equation}\label{eq: self-int-log-can}
(K_Z+\alpha \Gamma)^2=K_Z^2+d_\Gamma \alpha \left( 2(n-2)-n\alpha \right) \text{ and } 0 \leq \alpha<1-\varepsilon. 
\end{equation}

\begin{claim}\label{claim: bound_negative_curve}
The self-intersection of $\Gamma$ is bounded:
$$n \leq \frac{2}{\varepsilon}.$$
\end{claim}

\begin{proof}[Proof of Claim]\renewcommand{\qedsymbol}{$\blacksquare$}
By adjunction,
$$-2d_\Gamma=(K_Z+\Gamma) \cdot \Gamma= (K_Z+\alpha \Gamma) \cdot \Gamma +(1-\varepsilon -\alpha) \Gamma^2 +\varepsilon \Gamma^2.$$
As $-(K_Z+\alpha \Gamma)$ is big and nef and $1-\varepsilon> \alpha$, we have $(K_Z+\alpha \Gamma)\cdot\Gamma+(1-\varepsilon -\alpha) \Gamma^2 \leq 0$ and therefore we deduce $2 d_\Gamma \geq  \varepsilon d_\Gamma \cdot n$.
\end{proof}

If $K_Z \cdot \Gamma \leq 0$, then $(K_Z+\alpha \Gamma)^2 \leq K_Z^2 \leq 8$ by \autoref{lem: bound-self-intersection}, and we are done. So, assume $K_Z \cdot \Gamma >0$, or, equivalently, $n>2$. 
In this case, we have $d_\Gamma \leq m_\Gamma \leq 5$ by \cite[Proposition 4.7]{BT22}. 
Therefore by Equation \eqref{eq: self-int-log-can} and \autoref{claim: bound_negative_curve} we deduce the following series of inequalities:
\begin{equation*} \label{eq1}
\begin{split}
 (K_Z+\alpha \Gamma)^2 & = K_Z^2+d_\Gamma \alpha \left( 2(n-2)-n\alpha \right) \\
& \leq K_Z^2+ 5 \alpha \left( 2n-4\right)  \leq   K_Z^2+ 5(1-\varepsilon)\left(\frac{4}{\varepsilon}-4\right) \\ 
& \leq 8+20 \frac{(1-\varepsilon)^2}{\varepsilon}. \qedhere
	\end{split} 
\end{equation*} 
\end{proof}

As a consequence, we can show a boundedness result for klt del Pezzo surfaces of bounded Gorenstein index in characteristic $p>5$ (cf. \cite[Corollary 1.8]{HMX14} for the analogue in characteristic 0). In the following, we say that a klt del Pezzo surface is \emph{tame} if $h^1(X,\mathcal{O}_X) = 0$.

\begin{corollary}\label{cor: bounded-Gorenstein-index-dP}
Let $n>0$ be an integer.
Then, the classes
\begin{eqnarray*}
\mathcal{X}_{\textup{dP}, n}^{\textup{tame}} &=& \left\{ X \mid X \emph{ is a geometrically integral tame klt del Pezzo surface } \emph{s.t. } nK_X \emph{ is Cartier} \right\}, \text{ and} \\
\mathcal{X}_{\textup{dP}, n}^{>5} &=& \left\{ X \mid X \emph{ is a klt del Pezzo surface } \emph{s.t. } nK_X \emph{ is Cartier} \emph{ and }\emph{char}(H^0(X, \mathcal{O}_X)) \neq 2,3, 5 \right\}
\end{eqnarray*}
are bounded.
\end{corollary}

\begin{proof}
By \cite[Corollary 5.5]{BT22} and \cite[Theorem 5.7]{BT22}, klt del Pezzo surfaces in characteristic bigger than $5$ are geometrically integral and tame, so it suffices to show that $\mathcal{X}_{\textup{dP}, n}^{\textup{tame}}$ is bounded.

So, let $X \in \mathcal{X}_{\textup{dP}, n}^{\textup{tame}}$.
As $X$ is $\left(\frac{1}{n} \right)$-klt, \autoref{t-bounds-volume} implies that $K_X^2$ is bounded.
As the Cartier index of $K_X$ is fixed, the set of volumes $\left\{K_X^2 \mid X \in \mathcal{X}_{\textup{dP}, n}^{\textup{tame}}  \right\}$ is a finite set.
As $X$ has rational singularities by \autoref{p-klt-rational}, we can apply Riemann--Roch to compute for all $t \geq 1$:
$$\chi(X, \mathcal{O}_X(-ntK_X))=\chi(X, \mathcal{O}_X)+\frac{nt(nt+1)K_X^2}{2}.$$
As $X$ is tame, $\chi(X, \mathcal{O}_X)=1$ and therefore there are only a finite number of possibilities for the Hilbert polynomials $P_n(t):=\chi(X, \mathcal{O}_X(-ntK_X))$.
Finally we apply \cite[Theorem 2.1.2]{Kol85} to conclude that $X_{\overline{K}}$ form a bounded family over $\Spec(\mathbb{Z}[1/30])$. 
In particular, there exists $m:=m(n)$ such that $-mnK_{X_{\overline{K}}}$ is very ample, and thus, by faithfully flat descent, $-mnK_{X}$ is very ample.
Therefore, there exists $N=N(n)>0$ such $-mnK_{X}$ embeds $X$ into some $\mathbb{P}^N$ with a finite number of possibilities for the Hilbert polynomial. This concludes that $X$ belongs to a bounded family by a classical Hilbert scheme argument. 
\end{proof}

\subsection{Bounds on the $\mathbb{Q}$-Gorenstein index of $\varepsilon$-klt del Pezzo surfaces} \label{ss-bound-Qindex-eps}

After \autoref{cor: bounded-Gorenstein-index-dP}, to conclude the proof of boundedness of $\varepsilon$-klt del Pezzo surfaces we are only left to prove a bound on the Cartier index of $K_X$ depending only on $\varepsilon$.
We start with the following result, which is well-known over perfect fields.

\begin{lemma} \label{lem: bounds_Gorenstein_index}
    Fix $ \varepsilon \in \mathbb{Q}_{>0}$ and $ n \in \mathbb{Z}_{>0}$.
Then, there exists $N = N(\varepsilon,n)$ such that for every $\varepsilon$-klt surface $X$ admitting a minimal resolution $f \colon Y \to X$ with $\rho(Y) \leq n$, the divisor $NK_X$ is Cartier. 
\end{lemma}

\begin{proof}
    
    Without loss of generality, we assume that $X$ is the spectrum of a local ring and that $\rho(Y) = n$.
    Let $E= \sum_{i=1}^n E_i$ be the sum of the exceptional divisors of $f$ and write $K_Y + \sum_{i=1}^n b_i E_i =f^*K_X$ for some $0 \leq b_i< 1-\varepsilon$. By the base point free theorem \cite[Theorem 4.2]{Tan18b}, it suffices to find an effective $N=N(\varepsilon,n)$ such that $Nb_i$ is integral.
    For each $i$, we write $E_i^2 = -d_{E_i} n_i$ for some integer $n_i>0$, where $d_{E_i} = [H^0(E_i,\mathcal{O}_{E_i}):k]$. Moreover, for each $i \neq j$ we write $E_i \cdot E_j = d_{E_j} n_{ij}$ for some $n_{ij}>0$.
    The $b_i$ are determined by the following system of equations:
    \begin{equation} \label{eq: coefficients}
         (2 - n_j + b_jn_j) = \sum_{i\neq j} b_i n_{ij} \text{ for } j=1, \dots, n.
    \end{equation}
    Since $X$ is $\varepsilon$-klt, we have
    \begin{equation}\label{eq: bound_self_inters}
        n_j \leq \frac{2}{\varepsilon}.
    \end{equation} 
Indeed,
\begin{equation*} \label{eq:0}
\begin{split}
  0 & =(K_Y+\sum b_iE_i) \cdot E_j \geq (K_Y+b_j E_j) \cdot E_j \\
& \geq -2d_{E_j} + (b_j-1) \cdot (-d_{E_j}n_j) \geq -2d_{E_j}+\varepsilon n_jd_{E_j}. \\ 
\end{split}
\end{equation*}

    Moreover, $n_{ij}$ is bounded from above by \cite[Corollary 3.31 and Section 3.41]{kk-singbook} and \cite[Appendix A]{Sa23}. 
    As the $n_i$ and $n_{ij}$ are integers, we only have a finite number of possibilities for the coefficients in \autoref{eq: coefficients}, so we conclude that there are only finitely many possibilities for the solutions $b_i$, thus showing the existence of a $N(\varepsilon, n)$ for which $Nb_i$ is integral.
\end{proof}

\begin{lemma}\label{lem: bounds_sum_coeff}
    Let $X$ be a geometrically integral regular projective surface of del Pezzo type over $k$, and let $\pi \colon X \to B$ be a Mori fibre space.
    Let $\Delta=\sum b_i E_i$ be an effective $\mathbb{Q}$-divisor such that $(X,\Delta)$ is klt and $-(K_X+\Delta)$ is nef.
    Then 
    \begin{enumerate}
        \item if $\dim(B)=0$, then $\sum b_i \leq 3$;        \item if $\dim(B)=1$, then $\sum b_i \leq 4$.
    \end{enumerate}
\end{lemma}

\begin{proof}
    Suppose $\dim(B)=0$. Let $H$ be an ample Cartier divisor generating $\Num(X)$, and let $d \geq 0$ such that $-K_X \equiv dH$. 
    As $X$ is a regular del Pezzo surface, we have $K_X^2 \leq 9$ by \cite[Theorem 1.2]{Tan19} and therefore $d \leq 3$. Since $X$ is regular, the $B_i$ are Cartier divisors and thus $\sum b_i \leq 3$.

    Suppose $\dim(B)=1$. Let $F_b$ be a closed fibre of $\pi$ and $\Gamma$ the curve given by \autoref{lem: bound-self-intersection}.
    Set $d_b \coloneqq [k(b):k]$.
    We write $\Delta=b_0\Gamma + \sum b_iE_i$ as a sum of pairwise distinct prime divisors.
    As $0 \geq (K_X+\Delta) \cdot F_b$, adjunction implies 
    \begin{equation}
        2d_b \geq \Delta \cdot F_b \geq b_0d_b+\sum_{(E_i \cdot F_b) \neq 0} b_i d_b.
    \end{equation}
    As $K_X \cdot \Gamma=d_\Gamma(n-2)$, $0 \geq (K_X+\Delta) \cdot \Gamma$, and $b_0 \leq 1$, adjunction implies
    \begin{equation}
        2d_\Gamma \geq n  d_\Gamma -b_0 n d_\Gamma +\sum_i  b_i (E_i \cdot \Gamma) \geq \sum_{(E_i\cdot \Gamma ) \neq 0} b_i d_\Gamma.
    \end{equation}
    Summing the two equations and using that every curve on $X$ intersects either $\Gamma$ or $F_b$, we obtain
    $
    4 \geq \sum b_i,
    $
    as desired.
\end{proof}

\begin{proposition} \label{prop: bounds_Picard}
Let $\varepsilon >0$. Then, there exists a constant $D(\varepsilon)$ such that for all geometrically integral $\varepsilon$-klt del Pezzo surfaces $X$, 
the minimal resolution $f \colon Y \to X$ satisfies $\rho(Y) \leq D(\varepsilon).$

\end{proposition}

\begin{proof}
We can suppose $k$ is separably closed.
We follow the computations of \cite[Theorem 1.8]{AM04}, verifying that the explicit classification of rational Mori fibre spaces over algebraically closed fields is not needed.
Without loss of generality we can suppose $\varepsilon < \frac{2}{3}$.
Since $-K_X$ is ample, by Bertini's theorem we can choose $H \sim_{\mathbb{Q}} -K_X$ to be an effective $\mathbb{Q}$-divisor whose support is regular and contained in the regular locus of $X$ such that $(X,H)$ is $\varepsilon$-klt.
Let $f \colon Y \to X$ be the minimal resolution and write $K_Y+\Gamma_Y = f^*(K_X+H) \sim_{\mathbb{Q}} 0$ , where $\Gamma_Y \coloneqq \sum_{ i \in I} b_i E_i +f_*^{-1}H$ and $b_i < 1- \varepsilon$ by hypothesis. 
We run a $K_Y$-MMP which ends with 
$g \colon Y \to Z$, where $Z$ 
is a regular projective surface admitting a Mori fibre space structure $\pi \colon Z \to B$. We summarise the situation in the following diagram:

$$
\xymatrix{
 Y \ar[r]^{g} \ar^{f}[d] & Z \ar[d]^{\pi}   \\
 X & B
}
$$

We fix the following notation:
\begin{itemize}
\item $E_{i,Z} \coloneqq g_* E_i$.
\item $I_Z \coloneqq \{i \in I \mid E_{i,Z} \neq 0 \}$.
\item $\Delta_Y \coloneqq \sum_{i \in I} b_i E_i.$
\item $\Delta_Z \coloneqq \sum_{i \in I_Z} b_i E_{i,Z}$.
\item $\Gamma_Z \coloneqq g_*\Gamma_Y$
\end{itemize}

Note that, by construction, $g$ is a $(K_Y+\Delta_Y)$-non-positive (resp. $(K_Y+\Gamma_Y)$-trivial) birational contraction and $(Z, \Delta_Z)$ (resp. $(Z, \Gamma_Z)$) is a log del Pezzo pair (resp. Calabi--Yau pair) with $\varepsilon$-klt singularities.

The morphism $g$ is a composition of blow-ups of closed points on regular surfaces by \cite[\href{https://stacks.math.columbia.edu/tag/0C5R}{Tag 0C5R}]{stacks-project}. 
We can decompose $g$ as 
$$g \colon Y \xrightarrow[]{\psi} W  \xrightarrow[]{\varphi} Z ,$$ 
where
$\psi$ and $\varphi$ are proper birational morphisms between regular surfaces such that
\begin{enumerate}
    \item $\varphi$ is a composition of blow-ups at closed points $P$ such that $\text{mult}_P(\widetilde{\Gamma_Z})$ has multiplicity at least $\nu:=\frac{\varepsilon}{2}$, where $\widetilde{\Gamma_Z}$ denotes the strict transform of $\Gamma_Z$;
    \item $\psi$ is a composition of blow-ups at closed points $P$ where $\text{mult}_P(\widetilde{\Gamma_Z})$ has multiplicity $< \nu$.
\end{enumerate}

We first bound $\rho(W/Z)$ in terms of $\varepsilon$.
On $Y$, as $f_*^{-1}H$ is big and nef, we can apply Equation \eqref{eq: bound_self_inters} to obtain
\begin{equation}
(g_*^{-1}\Gamma_Z)^2 \geq \sum_{i \in I} b_i^2 (g_*^{-1} E_{i,Z})^2 \geq \sum_{i \in I_Z} b_{i} (1-\varepsilon)\left(\frac{-2}{\varepsilon}\right).
\end{equation}
If $\dim(B)=0$ (resp. $1$), we have $\sum_{i \in I_Z} b_i \leq 4$ (resp. $\leq 3$) by \autoref{lem: bounds_sum_coeff}, and thus
\begin{equation}\label{eq:ineq_selfint}
(g_{*}^{-1}\Gamma_Z)^2 \geq \begin{cases}
6-\frac{6}{\varepsilon} \text{ if } \dim(B)=0 \\
8-\frac{8}{\varepsilon} \text{ if } \dim(B)=1 .
\end{cases}
\end{equation}
After each of the blow-ups in $\varphi$, the self-intersection of $\widetilde{\Gamma_Z}$ decreases by at least $\nu^2$, we deduce that $(g_*^{-1}\Gamma_Z)^2 \leq \Gamma_Z^2-\nu^2 \rho(W/Z)$.
If $\dim(B)=0$ (resp. $1$), then $\Gamma_Z^2=K_Z^2 \leq 9$ (resp. $8$) by \cite[Theorem 1.2]{Tan19} (resp. \autoref{lem: bound-self-intersection}).
Therefore 
$ (g_*^{-1}\Gamma_Z)^2 \leq 9-\nu^2 \rho(W/Z)$ (resp. $8-\nu^2 \rho(W/Z)$).
Together with \eqref{eq:ineq_selfint} we conclude
 \begin{equation}\label{eq: first_bound_picard}
\rho(W/Z) \leq 
\begin{cases}
 \frac{(3\varepsilon + 6)}{\varepsilon \nu^2} \text{if } \dim(B)=0; \\
 \frac{8}{\varepsilon \nu^2} \text{ if } \dim(B)=1 .
\end{cases}
\end{equation}
As we chose $\varepsilon< \frac{2}{3}$ we have $\frac{(3\varepsilon + 6)}{\varepsilon \nu^2}< \frac{8}{\varepsilon \nu^2} $.

We now prove a bound on $\rho(Y/W)$ depending only on $\varepsilon$.  
Let $F=\sum_{i \in J} F_i$ be the sum of the exceptional divisors of $\varphi$ and let $f_i:=\text{coeff}_{F_i}(\Gamma_W)$, where $\Gamma_W \coloneqq \psi_* \Gamma_Y$. 
As $W$ and $Z$ are regular surfaces, $\text{Supp}(F)$ is an snc divisor.
Let $\psi \colon Y \xrightarrow{s} T \xrightarrow{t} W$ be a factorisation of $\psi$, where $t$ is a blow-up at a point $P$ of $W$ with exceptional divisor $C$.
Write $$K_T+\Gamma_T=K_T+\widetilde{\Gamma_Z} + \sum f_i\widetilde{F_i} +cC \sim_{\mathbb{Q}} t^*(K_W+\widetilde{\Gamma}_Z+\sum f_{i}F_{i}).$$
We claim that $P$ must lie on the intersection of two components $F_1$ and $F_2$ of $F$. 
For this, let $J_F \coloneqq \{i \in J \mid P \in F_i\}$.
% Suppose by contradiction that this is not the case.
Then, since $\text{mult}_P(\widetilde{\Gamma}_Z) \leq \nu < \varepsilon$ and $f_i < 1-\varepsilon$, we have that
$$0 < c= (\text{mult}_P(\widetilde{\Gamma}_Z)+\sum_{i \in J_F} f_i-1) \leq \nu + |J_F|(1-\varepsilon)-1,$$
hence $|J_F| \geq 2$.
Since ${\rm Supp}(F)$ is an snc divisor, this implies $|J_F| = 2$, as desired.
This argument can be repeated for each blow-up $T \to W$ factorising $Y \to W$.

As the number of nodes of $F$ is bounded by $\rho(W/Z)-1$, it remains to bound the number of times we are allowed to blow-up along a node to obtain a bound on $\rho(Y/W)$.
This follows from a straightforward induction as in \cite[Lemma 1.9]{AM04}. An explicit computation shows that
$$ \rho(Y)=\rho(Y/W)+\rho(W/Z)+\rho(Z) \leq \left(\frac{1}{(\varepsilon-\nu)^2}-1 \right) \cdot \left( \frac{8}{\varepsilon \nu^2} -1 \right) + \left( \frac{8}{\varepsilon \nu^2} \right) +2.$$
As $\nu=\frac{\varepsilon}{2}$, we deduce that
\begin{equation*}
\rho(Y) \leq \frac{128}{\varepsilon^5}+\left(3-\frac{4}{\varepsilon^2}\right) \leq \frac{128}{\varepsilon^5}. \qedhere
\end{equation*}
\end{proof}

We now have all the ingredients to prove the BAB conjecture in dimension 2 and characteristic $p \neq 2, 3 $ and $5$.

\begin{theorem} \label{thm: BAB_imperfect}
Let $\varepsilon>0$ be a rational number.
Then, the classes
\begin{eqnarray*}
\mathcal{X}_{\textup{dP}, \varepsilon}^{\textup{tame}} &=& \left\{ X \mid X \emph{ is a geometrically integral tame } \varepsilon\emph{-klt del Pezzo surface } \right\}, \text{ and} \\
\mathcal{X}_{\textup{dP}, \varepsilon}^{>5} &=&  \left\{ X \mid X \emph{ is an } \varepsilon\emph{-klt del Pezzo surface s.t. char}(H^0(X, \mathcal{O}_X)) \neq 2, 3, 5 \right\}
\end{eqnarray*}
are bounded.
\end{theorem}

\begin{proof}
    By \autoref{lem: bounds_Gorenstein_index} and \autoref{prop: bounds_Picard}, there exists $n=n(\varepsilon)>0$ such that $-nK_X$ is Cartier for all geometrically integral $\varepsilon$-klt del Pezzo surfaces $X$. 
    Hence we can apply \autoref{cor: bounded-Gorenstein-index-dP} to conclude that $\mathcal{X}_{\textup{dP}, \varepsilon}^{\textup{tame}} $ and $\mathcal{X}_{\textup{dP}, \varepsilon}^{>5}$ are bounded.
\end{proof}

\begin{remark}
    To prove the geometrically integral case of the BAB conjecture in characteristic $p \leq 5$, the missing ingredient is a bound on the irregularity for $\varepsilon$-klt del Pezzo surface. 
    While the canonical case (the characteristic $p >5$ klt case) has been treated in \autoref{t-intro-irr} (resp. \cite[Theorem 5.7]{BT22}), we are not able to prove a similar bound in the general case.
    Note that klt del Pezzo surfaces with $h^1(X, \mathcal{O}_X)=1$ are constructed in \cite{Tan20} over fields $k$ of characteristic $p=2, 3$ and $\pdeg(k)=1$.
\end{remark}

\bibliographystyle{amsalpha}
\bibliography{refs}

\end{document}